\documentclass[a4paper]{article}
\usepackage{geometry}
\usepackage{amssymb}
\usepackage{latexsym}
\usepackage{amsmath}
\usepackage{indentfirst}
\usepackage{graphicx}
\usepackage{enumerate}
\usepackage[colorlinks=true]{hyperref}
\usepackage{cancel}
\usepackage{mathrsfs} 

\newtheorem{Thm}{Theorem}[section]
\newtheorem{Cor}{Corollary}[section]
\newtheorem{Lem}{Lemma}[section]
\newtheorem{Pro}{Proposition}[section]

\newtheorem{Rek}{Remark}[section]
\newtheorem{Def}{Definition}[section]
\newcommand{\N}{\mathbb{N}}

\newcommand{\R}{\mathbb{R}}

\numberwithin{equation}{section} \numberwithin{figure}{section}

\newenvironment{proof}{\medskip\par\noindent{\bf Proof\/}:}{\qquad
\raisebox{-0.5mm}{\rule{1.5mm}{1mm}}\vspace{6pt}}

\begin{document}
\title{\Large\bf Normalized Solutions to the mixed fractional Schr\"odinger equations with potential and general nonlinear term}
\author{
{~~~~~~Anouar Bahrouni$^{1}$}\thanks{E-mail address: bahrounianouar@yahoo.fr}    {~~~~~~Qi Guo$^{2}$}\thanks{E-mail address: guoqi115@mails.ucas.ac.cn}   {~~~~~~Hichem Hajaiej$^{3}$}\thanks{Corresponding author: hhajaie@calstatela.edu}\\
\small $^1\,$Mathematics Department, University of Monastir,
Faculty of Sciences, 5019 Monastir, Tunisia \\
\small  $^2\,$School of Mathematics, Renmin University of China, Beijing 100086, P. R. China
\\
\small $^3\,$Department of Mathematics, California State University, Los Angeles,\\
\small 5151 State University Drive,
Los Angeles, CA 90032, USA\\
}

\date{} \maketitle
%
{\small  \noindent{\bf Abstract:} The purpose of this paper is to
establish the existence of solutions with prescribed norm to a class
of nonlinear equations involving the mixed fractional Laplacians.
This type of equations arises in various fields ranging from
biophysics to population dynamics. Due to the importance of these
applications, this topic has very recently received an increasing
interest. This work extends the results obtained in \cite{ding,
hichem}
 to the mixed fractional Laplacians. Our method is novel and our results cover all the previous ones.
}
\section{Introduction}
This paper concerns the existence of solutions $(u_a,\lambda_a)\in H^{s_1,s_2}(\mathbb{R}^d)\times \mathbb{R}$ to the
following fractional equation
\begin{align}\label{maineq}
\begin{cases}
&(-\Delta)^{s_1}u(x)+(-\Delta)^{s_2}u(x)+\lambda u(x)+V(x)u(x)=g(u(x)),\  x\in \mathbb{R}^{d},\\
&\int_{\mathbb{R}^d} |u(x)|^2dx=a,
\end{cases}
\end{align}
where $0<s_1<s_2<1$,  $2s_1<d<\frac{2s_1s_2}{s_2-s_1}$, $a>0$. The precise condition on $g$ and $V$ will be given
later on. The fractional Laplacian is given by
$$(-\Delta)^{s_i}u(x)=C_{d,s_{i}}\displaystyle \int_{\mathbb{R}^{d}}\frac{u(x)-u(y)}{|x-y|^{d+2s_i}}dy, \ \ i=1,2,$$
with
$C_{d,s_{i}}:=2^{2s_i}\pi^{-\frac{d}{2}}s_i\frac{\Gamma(\frac{d+2s_i}{2})}{\Gamma(1-s_i)}$,
where $\Gamma$ is the Gamma function, see \cite{gamma}.

Our interest in this problem results from \cite{hichem}. The authors
were interested in finding normalized solutions for a class of mixed
fractional equations with simplified potentials. Equation
\eqref{maineq} arises in the superposition of two stochastic process
with a different random walk and a L\'evy flight. The associated
limit diffusion is described by a sum of two fractional Laplacian
with different orders, see \cite{Ssdipierro}. Very recently,
Dipierro et al. in \cite{Sdipierro} tackled another interesting
problem related to \eqref{maineq}. In particular, it turns out that
the mixed fractional Laplacians models the population dynamics, some
heart anomalics caused by artries issues. Those heart problems can
be modeled thanks to the superposition of two to five mixed
fractional Laplacians since it is not necessarily the same anomaly
in the five artries, see \cite{Elshahed}. Equation \eqref{maineq}
also plays a crucial role in other fields, i.e. chemical reaction
design, plasma physics, biophysics, see \cite{Aris, Dipierro,
Wilhelmsson}. The $L^2-$norm is a preserved quantity of the
solutions of the dynamic version of \eqref{maineq}. Moreover, the
normalized solutions are known to provide stable solutions. This is
very attractive for applications and lead many research groups to
focus on this area. The study of fractional problems has gained a
lot of interest after the
    pioneering papers  of Caffarelli et al.
    \cite{Caf1,Caf2,Caf3}. Inspired with
    this, several other works have been published in the nonlocal
    framework, see for instance, \cite{Bahrouni,Radulescu,Bucur,Dipierro,valdinoci}.


There are two different ways to deal with equation \eqref{maineq}
according to the role of $\lambda$:\\
(i) the frequency $\lambda$ is a fixed and assigned parameter;
\\
(ii) the frequency $\lambda$ is considered as an unknown in the problem.\\
For case (i), one can apply variational methods looking for critical
points of the following functional
$$I(u)=\frac{1}{2}\int_{\R^d}|\nabla_{s_1}u|^{2}dx+\frac{1}{2}\int_{\R^d}|\nabla_{s_2}u|^{2}dx+\frac{\lambda}{2}\int_{\R^d}V(x)u^2 dx-\int_{\R^d} G(u)dx,$$
where $G(s):=\int_{0}^{s}g(t)dt,$ for $s\in \R$. Also, we can take
advantage of some other topological methods such as the fixed point
theory, bifurcation or the Lyapunov-Schmidt reduction. A large
number of papers are devoted to study of this kind of problems and
it is impossible to summarize it here since the related literatures
are huge. Alternatively, one can search for solutions to equation
\eqref{maineq} with the frequency $\lambda$ unknown. In this case,
the real parameter $\lambda$ appears as a Lagrange multiplier and
$L^2-$norm of solutions are prescribed.  They are usually called
normalized solutions. Let us introduce and review some important
works in this direction that we are going to use in some way in this
work.\\

In \cite{jeanjean}, Jeanjean  studied the existence  of solutions
with prescribed norm in the framework of semilinear elliptic
equations. In particular, he studied equation \eqref{maineq} where
$s_1=s_2=1$ and $V=0$. In order to overcome the lack of compactness
the author worked in the radial Sobolev space to get some
compactness results. Moreover, he considered the nonradial case by
using a characterization of the Palais-Smale sequence introduced in
\cite{jeanjean2019}. However, he did not work directly with the
functional associated to the problem. In his approach, he considered
a new modified functional to simplify the obtention of Palais-Smale
sequence at a suitable mountain pass level. His ideas were used by
many other papers, but they don't apply to our complex situation.
Further details will be pointed later.

We recall that the number $\overline{p}:=2+\frac{4}{d}$ is called in
the literature as $L^2-$critical exponent, which comes from the
Gagliardo-Nirenberg inequality , see \cite{boulenger}. It is worth
to mention that in \cite{Soave1}, Soave studied the existence of
normalized solutions for the $L^2-$subcritical (that is $q\in
(2,\overline{p}))$ nonlinear Schr\"odinger equation with combined
power nonlinearities of the type
$$g(t)=\mu |t|^{q-2}t+|t|^{p-2}t, \ \ \mbox{with} \ \ \mu>0.$$
In \cite{jeanjean1}, the authors discussed the existence of ground
state normalized solution for the $L^2-$subcritical case while in
\cite{jeanjean3} a multiplicity result is established such that the
second solution is not a ground state. We also refer to
\cite{bartsch,soave2,yang} for the existence of normalized solutions
in the $L^2-$supercritical case. In \cite{bartsh1}, the authors
studied the existence of infinitely many normalized solutions.

When $s_1=s_2=1:$ For the potential case, that is $V\neq 0$,
Pellacci et al. \cite{Pellacci} applied Lyapunov-Schmidt reduction
approach to study problem \eqref{maineq} for the special case where
$g(u)=u^{p}$. In \cite{Ikoma}, Ikoma et al treated the potential
case with general nonlinearity. The case of positive potential and
vanishing at infinity is considered in the very recent paper
\cite{bartsch2}. In such a case, the mountain pass structure
introduced in \cite{jeanjean} is destroyed. By constructing a
suitable linking structure, the authors proved the existence of
normalized solutions with high Morse index. The case of negative
potential is considered in \cite{Molle} with the particular
nonlinearity $g(u)=|u|^p$. Very recently, in \cite{ding} Ding et al.
treated the case of negative potential and a more general
nonlinearity. In \cite{yang2}, Yang et al. studied the existence and
multiplicity of normalized solutions to the Schr\"odinger equations
with potentials and non-autonomous nonlinearities. For important
contributions to the study of normalized solutions we refer to the
works of Hajaiej and Stuart \cite{hs1,hs2}.

Contrary to the local case, the situation seems to be in a
developing state for the fractional Laplacian, see
\cite{Luo,Radulescunor}. It is worth mentioning that there is only
one paper devoted to the study of normalized solution for nonlinear
equations in which the mixed fractional Laplacians is present. In
\cite{hichem}, Hajaiej et al considered the following equation
$$ (-\Delta)^{s_1}u(x)+(-\Delta)^{s_2}u(x)+\lambda u(x)=|u|^{p}u,\  x\in \mathbb{R}^{d} $$
where $s_1<s_2$, $p,d>0$ and $\lambda\in \R.$ The authors discussed
the existence and nonexistence of normalized solution for the above
problem.

The novelty of our work comes from the presence of a general
potential and nonlinearity in Eq. \eqref{maineq}.  This is not an
easy situation to deal with for the mixed fractional Laplacians. To
the best of our knowledge, this is the first paper proving the
existence of normalized solution in this very general context. More
precisely, the main difficulties that arise in treating this problem
are the following: (i) the lack of compactness due to the
unboundedness of the domain; (ii) the absence of the fundamental
properties of the mixed fractional Laplacians such us the
classification of the Palais Smale sequences; (iii) the presence of
the nonradial potential $V$ which forbid the use of the principle of
symmetric criticality and perturb the arguments employed for the
nonpotential case. To avoid the two first difficulties, we prove a
new theorem in which we give a classification of the Palais
sequences (see Theorem \ref{Palais}).
The main key to overcome the third difficulty is to exploit some
properties of the functional associated to equation \eqref{maineq}
with $V=0$. In particular, we will show that there exists $u\in
H^{s_1,s_2}(\mathbb{R}^{d})$, such that
$$m_{a}=\displaystyle \inf_{v\in \mathcal{P}_{\infty,a}}I(v)=I(u),$$
where $I$ and $\mathcal{P}_{\infty,a}$ are defined in Section $3$.
There are two ways to
 prove that $m_a$ is attained; The first idea to establish this result is to work on the radial
 space to get some compactness properties (see Subsection $3.2$). However, the fact that $m_a$ is attained for the
 radial space is not enough to prove the existence of normalized
 solution for problem \eqref{maineq} with potential, that is, $V\neq 0$. For this
 reason, we will give a new methods which combine some new technical lemmas with the argument employed by
 Jeanjean in
 \cite{jeanjean}. 
 \medskip

 For the reader's convenience, we now state
all the
conditions on $g$ and $V$:
\begin{itemize}
        \item[$(G_1)$] $g:\mathbb{R}\rightarrow \mathbb{R}$ is
        continuous and odd.
        \item[$(G_2)$] 
            We assume that there exits  $(\alpha,\beta)\in \mathbb{R}_{+}^{2}$
        satisfying
        $$2+\frac{4s_2}{d}<\alpha<\beta<\frac{2d}{d-2s_1},$$
        such that
        $$\alpha G(s)\leq g(s)s\leq \beta G(s), \ \ \forall s\in \mathbb{R},\ \ \mbox{with} \ \ G(s)=\displaystyle \int_{0}^{s}g(t)dt.$$
\item[$(G_3)$] Let $\widetilde{G}:\mathbb{R}\rightarrow \mathbb{R},$
$\widetilde{G}(s)=\frac{g(s)s}{2}-G(s)$. We assume that $\widetilde{G}^{'}$
exists and
$$\widetilde{G}^{'}(s)s\geq \alpha \widetilde{G}(s), \ \ \forall s\in \mathbb{R},$$
where $\alpha$ is given by $(G_2)$.
\item[$(V_1)$]   $\lim\limits_{|x|\rightarrow +\infty} V(x)=\sup\limits_{x\in \mathbb{R}^d}V(x)=0$ and there exists some
$\sigma_{1}\in \left[0,\frac{d(\alpha-2)-4}{d(\alpha-2)}\right]$ such that
$$\displaystyle \left|\int_{\mathbb{R}^{d}}V(x)u^{2}dx\right|\leq \sigma_1\left(|\nabla_{s_1}u|_{2}^{2}+|\nabla_{s_2}u|_{2}^{2}\right), \ \ \forall u\in H^{s_1,s_2}(\mathbb{R}^{d}).$$

\item[$(V_2)$] $\nabla V(x)$ exists for a.e $x\in \mathbb{R}^{d}$.
Set $W(x)=\frac{1}{2}\langle\nabla V(x),x\rangle$. We assume that
$\lim\limits_{|x|\rightarrow +\infty} W(x)=0$, and there exists a
positive constant $\sigma_{2}$ such that
$$0<\sigma_{2}<\min\left\{s_1-\frac{(\beta-2)d}{2\beta},\frac{d(\alpha-2)(1-\sigma_1)}{4}-s_2\right\}$$ such that $$\left|\int_{\mathbb{R}^{d}}W(x)u^{2}dx\right|\leq
\sigma_2\left(|\nabla_{s_1}u|_{2}^{2}+|\nabla_{s_2}u|_{2}^{2}\right),
\ \ \forall u\in H^{s_1,s_2}(\mathbb{R}^{d}),$$ where $\alpha,
\beta$ and $\sigma_1$ are defined in $(G_1)$ and $(V_1)$.\\
\item[$(V_3)$] $\nabla W(x)$ exists for a.e $x\in \mathbb{R}^{d}$.
Put $Y(x)=(d\alpha/2-d-1)W(x)+\langle\nabla
W(x),x\rangle.$ We suppose that, there exists some $\sigma_3\in
[0,s_{1}^{2}(d\alpha/2-d-2s_2)]$ such that
$$\left|\int_{\mathbb{R}^{d}}Y(x)u^{2}dx\right|\leq
\sigma_3\left(|\nabla_{s_1}u|_{2}^{2}+|\nabla_{s_2}u|_{2}^{2}\right), \ \
\forall u\in H^{s_1,s_2}(\mathbb{R}^{d}).$$
\end{itemize}
As a consequence of the above assumptions, we can deduce the
following two remarks, see \cite{jeanjean, yang}.
\begin{Rek}\label{rem}
$(1)$ It immediately follows from $(G_1)$ and $(G_2)$ that, for all
$t\in \mathbb{R}$ and $s\geq 0$,
\begin{equation*}
\begin{cases}
s^{\beta}G(t)\leq G(ts)\leq s^{\alpha}G(t), \ \ s\leq 1,\\
s^{\alpha}G(t)\leq G(ts)\leq s^{\beta}G(t), \ \ s\geq 1.
\end{cases}
\end{equation*}
$(2)$ There exist some $C_1,C_2>0$ such that, for all $s\in
\mathbb{R}$,
\begin{equation*}
\begin{cases}
C_1\min(|s|^{\alpha},|s|^{\beta})\leq G(s)\leq C_2
\max(|s|^{\alpha},|s|^{\beta})\leq C_2(|s|^{\alpha}+|s|^{\beta}),\\
(\frac{\alpha}{2}-1)G(s)\leq \widetilde{G}(s)\leq
(\frac{\beta}{2}-1)G(s)\leq (\frac{\beta}{2}-1) C_2
(|s|^{\alpha}+|s|^{\beta}).
\end{cases}
\end{equation*}
\end{Rek}
\begin{Rek}
$(1)$ Set $g(t)=|t|^{p-2}t+|t|^{q-2}t$, for $t\in \mathbb{R}$, such
that
$$2+\frac{4s_2}{d}<p<q<\frac{2d}{d-2s_1}.$$ A trivial verification
shows that
$(G_1)-(G_3)$ are satisfied. \\
$(2)$ By Sobolev inequality, under some small conditions on
$|V|_{\frac{N}{2}}, |W|_{\frac{N}{2}}$ and $|Y|_{\frac{N}{2}}$, one
can easily check that $(V_1)-(V_3)$ are satisfied.
\end{Rek}
Our main result reads as follows.
\begin{Thm}\label{mainnthm}
Suppose that assumptions $(G_1)-(G_3)$ and $(V_1)-(V_3)$ are
satisfied.  Then\\
$(1)$ Problem \eqref{maineq} has no nontrivial solution $u\in
H^{s_1,s_2}(\R^d)$ provided that $\lambda \leq 0$.\\
$(2)$
For any $a>0$, there
exists a couple $(\lambda_a,u_a)\in \mathbb{R}^+\times
H^{s_1,s_2}(\R^d)$ solves \eqref{maineq}.
\end{Thm}
Next, we consider case where the potential $V(x)$ satisfies the
following coercive-type assumptions:
\begin{itemize}
\item[$(V'_1)$] $V(x)\in C(\R^d,\R)$ with $V_0:=\inf_{x\in \R^d} V(x)>0$.
\item [$(V'_2)$] For every $M>0$,
 \[\mu(\{x\in \R^d: V(x)\leq M\})<\infty,\]
where $\mu$ denotes the Lebesgue measure in $\R^d$.
\end{itemize}
Then, we give our second main result.
\begin{Thm}\label{mainnthm2}
 Suppose that assumptions $(G_1)-(G_3)$ and $(V'_1)-(V'_2)$ are satisfied.  Then for any $a>0$, there exists a couple $(\lambda_a,u_a)\in\mathbb{R}^+\times H^{s_1,s_2}(\R^d)$ solves \eqref{maineq}.
\end{Thm}

The paper is organized as follows. In Section $2$, we state the main
notations and the main results that will be used later. The basic
properties of the associated functional to the auxiliary problem
with $V=0$ are then considered in Section $3$. 
In Sections $4$ and $5$, we will show the proof of Theorems
\ref{mainnthm} and \ref{mainnthm2}.

\section{Preliminary results}

In this section, we will focus on the variational structure of
equation \eqref{maineq}. First we introduce the fractional Sobolev
space
$$H^{s}(\mathbb{R}^d)=\{u\in L^2(\mathbb{R}^d):|\nabla_s u|_2<+\infty\},$$
with $$|\nabla_s u|_2^2=\int_{\mathbb{R}^d\times \mathbb{R}^d}\frac{|u(x)-u(y)|^2}{|x-y|^{d+2s}}dxdy.$$
\medskip
\\
For $0<s_1<s_2<1$, we set
 $$H^{s_1,s_2}(\R^d)=H^{s_1}(\R^d)\cap
H^{s_2}(\R^d).$$ One can check that $H^{s_1,s_2}(\R^d)$ is a Hilbert
space  with respect to the scalar product
\[\langle u,v\rangle_{H^{s_1,s_2}}=\int_{\R^d} u(x)v(x)dx+\frac{1}{2}\sum\limits_{i=1}^{2}\int_{\R^d\times \R^d} \frac{(u(x)-u(y))v(x)}{|x-y|^{d+2s_i}}dxdy,\ \forall u,v\in H^{s_1,s_2}(\R^d),\]
with the induced norm
$$\|u\|^{2}=\|u\|_{H^{s_1,s_2}}^{2}=|\nabla _{s_{1}}u|_2^{2}+|\nabla _{s_{2}}u|_2^{2}+|u|_{2}^{2}.$$
\\
Denote $2^*_s=\frac{2d}{d-2s}$. It is well known that $H^{s}(\mathbb{R}^d)$ is continuously embedded
in $L^{q}(\R^d)$, for any $q\in [2, 2^*_{s }].$ Then
\[H^{s_1,s_2}(\R^d)\subset L^q(\R^d), \ \forall q\in [2,2^*_{s_2}].\]
And we denote $H^{s_1,s_2}_{r}(\mathbb{R}^{d})$ by
$$H^{s_1,s_2}_{r}(\mathbb{R}^{d})=\{u\in H^{s_1,s_2}(\mathbb{R}^{d}) : u(x)=h(|x|), \ \ h:[0,+\infty[\rightarrow \mathbb{R} \}.$$
\begin{Lem}\label{rad}
$H^{s_1,s_2}_{r}(\mathbb{R}^{d})$ is compactly embedding into
$L^{p}(\mathbb{R}^{d})$ for $p\in (2,2^{\ast}_{s_{2}})$.
\end{Lem}
\begin{proof}
The proof is similar to that in \cite{lions}.
\end{proof}
\\
Recall that from \cite{boulenger}, if $0<s<1$, $d>2s$,
$p<\frac{2d}{d-2s}$, then the fractional Gagliardo-Nirenberg
inequality
\begin{equation}\label{gagliardo}
|u|_{p+2}^{p+2}\leq B(p,ds)
|\nabla_{s}u|_{2}^{\frac{dp}{2s}}|u|_{2}^{p+2-\frac{dp}{2s}}, \ \
\forall u\in H^{s_1,s_2}(\R^d),
\end{equation}
holds with the optimal constant $B(d,p,s)>0$ given by
$$B(d,p,s)=\left(\frac{2s(p+2)-dp}{dp}\right)^{\frac{dp}{4s}}\frac{2s(p+2)}{\left(2s(p+2)-dp\right)\|\mathcal{Q}_s\|_{2}^{p}},$$
and $\mathcal{Q}_s$ is the ground state of
$$(-\Delta )^{s}u+u=|u|^{p}u, \ \ x\in \R^{d},$$
whose existence and uniqueness have been proved by \cite{frank}.
 For more details on fractional Sobolev spaces, we refer the
  readers to see \cite{Radulescu,valdinoci}.\\
Now, we recall the definition of the Palais Smale sequence.
\begin{Def}
Let $I\in C^1(H^{s_1,s_2}(\R^d),\R)$, we say that the sequence $\{u_n\}\subset  H^{s_1,s_2}(\R^d)$ is a $(PS)_c$-sequence
for $I$ at level $c\in \R$ if
$$I(u_n)\rightarrow c \ \ \mbox{and} \ \ I'(u_n) \rightarrow 0 \ \ \mbox{in} \ \ \left(H^{s_1,s_2}(\R^d)\right)^{*}. $$
\end{Def}
 A solution $u$ to the problem \eqref{maineq} with $\displaystyle \int_{\R^d}
 |u(x)|^{2}dx=a $ corresponds to a critical point of the following
 $C^1$-functional
\[J(u)=\frac{1}{2}|\nabla_{s_1}u|_2^2+\frac{1}{2}|\nabla_{s_2}u|_2^2+\frac{1}{2}\int_{\R^d}V(x)u^2dx-\int_{\R^d}G(u)dx,\]
restricted to the sphere in $L^{2}(\R^d)$ given by  $$S_a:=\{u\in
H^{s_1, s_2}(\R^d): \int_{\R^d} |u(x)|^2dx=a\}.$$ Then the parameter
$\lambda$ appears as Lagrange multiplier.\\
The corresponding functional to \eqref{maineq} when $V=0$ is
\[I(u)=\frac{1}{2}|\nabla_{s_1}u|_2^2+\frac{1}{2}|\nabla_{s_2}u|_2^2 -\int_{\R^d}G(u)dx, \ \ \forall u\in H^{s_1,s_2}(\R^{d}).\]
It is clear that, under assumptions $(G_1)-(G_3)$, the functional
$I$ is of class $C^{1}$.\\
In order to follow the same strategy of \cite{ding}, we need the
following definitions to introduce our variational procedure.\\
$(1)$ The fiber map,
$$u(x)\rightarrow (t*u)(x)=t^\frac{d}{2}u(tx),$$ for $(t,u)\in
\mathbb{R}^{+}\times H^{s_1,s_2}(\mathbb{R}^{d})$, which preserves
the $L^{2}-$norm.\\
$(2)$ The modified functionals associated to equations
\eqref{maineq} and nonpotential case of \eqref{maineq} are
respectively
$$\Psi_{\infty,u}(t)=I(t*u), \ \ \forall (t,u)\in \R \times H^{s_1,s_2}(\R^d)$$
and
$$\Psi_{u}(t)=J(t*u), \ \ \forall (t,u)\in \R \times H^{s_1,s_2}(\R^d)$$
 $(3)$  Let
$$P_\infty(u):=s_1|\nabla_{s_1}u|_2^2+s_2|\nabla_{s_2}u|_2^2-d\int_{\R^d}\widetilde{G}(u)dx,$$
 $$\mathcal{P}_\infty:=\{u\in H^{s_1,s_2}(\R^d):P_\infty (u)=0\},
\ \  \mathcal{P}_{\infty,a}=S_a\cap \mathcal{P}_\infty,$$ and
$$\mathcal{P}_{\infty,a,r}=\mathcal{P}_{\infty,a}\bigcap H^{s_1,s_2}_r(\R^d).$$
 $(4)$ We introduce the following  function:
\[P(u)=s_1|\nabla_{s_1}u|_2^2+s_2|\nabla_{s_2}u|_2^2-\frac{1}{2}\int_{\R^d}\langle \nabla V(x),x\rangle
u^2dx-d\int_{\R^d} \widetilde{G}(u)dx.\]\\
$(5)$ We set
$$\mathcal{P}:=\{u\in H^{s_1,s_2}(\R^d):P (u)=0\},
\ \  \mathcal{P}_{a}=S_a\cap \mathcal{P}.$$

\section{Properties of the  functional $\text{I}$}

In this section, we explore the properties of  the functional
associated to the auxiliary problem \eqref{maineq} without linear
potential. More precisely, we will  prove that
$m_a:=\inf\limits_{u\in \mathcal{P}_{\infty,a }} I(u)$ and
$m_{a,r}:=\inf\limits_{u\in \mathcal{P}_{\infty,a,r}} I(u)$ are
attained. Here, we would like to mention that since the functional
$I$ is considered on the whole space $\R^N$, the qualitative
properties of $I$ are more difficult to establish due to the lack of
compactness. For this reason, we will consider two cases: the radial
and the nonradial cases.

\subsection{Technical Lemmas}
In this subsection, we prove some technical lemmas which will be
useful to prove the main result of Theorem
\ref{dec}. We need to use the radial Sobolev space only in the
proof of Theorem \ref{dec}. So
the rest of Lemmas are considered in $H^{s_1,s_2}(\R^d)$. We establish our results in the spirit of the papers \cite{bartsch,yang}.

First, we  observe that
\begin{align*}
\Psi_{\infty,u}'(t)=s_1 t^{2s_1-1} |\nabla_{s_1}u|_2^2+s_2
t^{2s_2-1} |\nabla_{s_2}u|_2^2-\frac{d}{t^{d+1}}\displaystyle
\int_{\mathbb{R}^{d}}\widetilde{G}(t^{\frac{d}{2}}u(x))dx=\frac{1}{t}P_\infty(t*u)
\end{align*}
for all $(t,u)\in \R\times H^{s_1,s_2}(\R^{d})$.
\begin{Lem}\label{max}
Suppose that $(G_1)-(G_3)$ hold. Then,  for any $u\in S_a$,  there
exists a unique $t_u>0$ such that $t_u*u\in \mathcal{P}_{\infty,
a}$. Moreover,
\[I(t_u*u)=\max\limits_{t>0}I(t*u).\]
Consequently,
\[m_a:=\inf\limits_{u\in S_a}\max\limits_{t>0} I(t*u)=\inf\limits_{u\in \mathcal{P}_{\infty,a}} I(u).\]
\end{Lem}
\begin{proof}
First, we show that there exists (at least) a $t_u\in \mathbb{R}^+$,
such that $t_u*u\in \mathcal{P}_{\infty,a}$. This follows directly
from assumptions $(G_1)$ and $(G_2)$. Recall that
\begin{align*}
P_{\infty}(t*u)=s_1t^{2s_1}|\nabla_{s_1}u|_2^2+s_2t^{2s_2}|\nabla_{s_2}u|_2^2-\frac{d}{t^{d}}\int_{\R^d}\widetilde{G}(t^{\frac{d}{2}}u(x))dx.
\end{align*}
In light of Remark \ref{rem}, we can deduce that
\begin{align}\label{3.1}
C_1t^{2s_1}+C_2t^{2s_2}-C_3\max\{t^{\frac{\alpha-2}{2}d},t^{\frac{\beta-2}{2}d}\}\leq
P_\infty (t*u)\leq
C_1t^{2s_1}+C_2t^{2s_2}-C_3\min\{t^{\frac{\alpha-2}{2}d},t^{\frac{\beta-2}{2}d}\},
\end{align}
where $C_1, C_2,C_3>0$ are independent of $t$. Hence, for $t$ large
enough and using \eqref{3.1}, we obtain
\begin{align*}
 P_\infty
(t*u)\leq C_1t^{2s_1}+C_2t^{2s_2}-C_3t^{t^{\frac{\alpha-2}{2}d}},
\end{align*}
which proves that
\begin{align}\label{3.2}
\lim\limits_{t\rightarrow +\infty}P_\infty(t*u)=-\infty.
\end{align}
Again, for $t$ small enough and using \eqref{3.1}, we get
\begin{align*}
C_1t^{2s_1}+C_2t^{2s_2}-C_3t^{\frac{\alpha-2}{2}d}\leq P_\infty
(t*u)\leq C_1t^{2s_1}+C_2t^{2s_2}-C_3t^{\frac{\beta-2}{2}d},
\end{align*}
which implies that
\begin{align}\label{3.3}
\lim\limits_{t\rightarrow 0+}P_\infty(t*u)=0 \ \ \mbox{and} \ \
P_\infty(t*u)>0, \ \ \mbox{for} \ \ t \ \ \mbox{small enough}.
\end{align}
 In the above inequalities, we used the fact that $2s_1<2s_2<\frac{\alpha-2}{2}d<\frac{\beta-2}{2}d$.  Consequently, by combining \eqref{3.2} and \eqref{3.3}, there must
exist $t_u>0$, such that $P_\infty(t_u*u)=0$. Due to
$\Psi'_{\infty,u}(t)=\frac{1}{t} P_\infty(t*u)$, we have $t_u$ is a
critical point of $\Psi_{\infty,u}$. Therefore, using this fact and
$t*u\in S_a$, we can infer that $t_u*u\in \mathcal{P}_{\infty, a}$.
\\
Next, we show that $t_u$ is unique and $\Psi_{\infty,u}$ reaches its
maximum at $t_u$. To reach this goal, we only need to show that for
any critical point
 $t_u$ of $\Psi_{\infty,u}(t)$, we have $\Psi_{\infty,u}''(t)|_{t=t_u}<0$.\\
Let $t_u$ be a critical point of $\Psi_{\infty,u}$. Then
\begin{align}\label{3.4}
s_1t_u^{2s_1}|\nabla_{s_1}u|_2^2+s_2t_u^{2s_2}|\nabla_{s_2}u|_2^2-\frac{d}{t_u^{d}}\int_{\R^d}\widetilde{G}(t_u^{\frac{d}{2}}u(x))dx=0.
\end{align}
Now, by direct computation, we have
\begin{align}\label{3.5}
\Psi''_{\infty,u}(t)=&s_1(2s_1-1)t^{2s_1-2}|\nabla_{s_1}u|_2^2+s_2(2s_2-1)t^{2s_2-2}|\nabla_{s_2}u|_2^2
+\frac{d(d+1)}{t^{d+2}}\int_{\R^d}\widetilde{G}(t^{\frac{d}{2}}u(x))dx\\
&-\frac{d^2}{2t^{\frac{d}{2}+2}}\int_{\R^d}\widetilde{G}'(t^{\frac{d}{2}}u(x))\cdot
u(x)dx. \nonumber
\end{align}
Therefore, from \eqref{3.4} and \eqref{3.5},  we infer that
\begin{align}\label{3.6}
\Psi''_{\infty,u}(t)|_{t=t_u}=2(s_1-s_2)s_1t_u^{2s_1-2}|\nabla_{s_1}u|_2^2+\frac{(2s_2+d)d}{t_u^{d+2}}\int_{\R^d}\widetilde{G}(t_u^{\frac{d}{2}}u)dx
-\frac{d^2}{2t_u^{d+2}}\int_{\R^d}\widetilde{G}'(t_u^{\frac{d}{2}}u)\cdot
t_u^\frac{d}{2}u dx.
\end{align}
On the other hand, by assumption $(G_3)$, we obtain
\begin{align}\label{3.7}
\frac{(2s_2+d)d}{t_u^{d+2}}\int_{\R^d}\widetilde{G}(t_u^{\frac{d}{2}}u)dx
-\frac{d^2}{2t_u^{d+2}}\int_{\R^d}\widetilde{G}'(t_u^{\frac{d}{2}}u)\cdot
t_u^\frac{d}{2}u dx\leq
(2s_2+d-\frac{d\alpha}{2})\frac{d}{t_u^{d+2}}
\int_{\R^d}\widetilde{G}(t_u^{\frac{d}{2}}u)dx.
\end{align}
Consequently, by the fact that $\alpha>2+\frac{4s_2}{d}$,
$s_1<s_2$ and using \eqref{3.6} and \eqref{3.7}, we conclude that
$$\Psi''_{\infty,u}(t)|_{t=t_u}<0.$$
This ends the proof.
\end{proof}

 \begin{Rek}
 From the proof of Lemma \ref{max}, we have that if $P_{\infty}(u)\leq 0$ ($P_{\infty,r}(u)\leq 0$), then $\exists t_u\in (0,1]$, such that $t_u*u\in \mathcal{P}_{\infty,a}$ ($t_u*u\in \mathcal{P}_{\infty,a,r}$).
 \end{Rek}

\begin{Lem}\label{decc}
Suppose that $(G_1)-(G_3)$ hold and consider the function
$m:\mathbb{R}^+\rightarrow \R$ defined by $m(a):=m_a$, then $m$ is
strictly decreasing.
\end{Lem}

\begin{proof}
Let $0<a_1<a_2<\infty$ and $\theta=\frac{a_2}{a_1}>1$. For any
$u_1\in S_{a_1}$, set
\[ u_2(x)=\theta^A u_1(\theta^B x),\]
where $A=\frac{2s_1-d}{4s_1}<0$, $B=-\frac{1}{2s_1}<0$. It is clear
that
\[2A-B d=1, \ 2A +B(2s_1-d)=0.\]
This implies that
\begin{align*}
|\nabla_s(\theta^A u(\theta^B
x))|_2^2=\int_{\R^d\times\R^d}\frac{\theta^{2A}|u(\theta^B
x)-u(\theta^B y)|^2}{|x-y|^{d+2s}}dxdy =\theta^{2A+B(2s
-d)}|\nabla_s u|_2^2,
\end{align*}
and
\[ |\theta^A u(\theta^B x)|_2^2=\theta^{2A-B d}|u|_2^2.   \]
Therefore
\[ |\nabla_{s_1}u_2(x)|_2^2=|\nabla_{s_1}u_1(x)|_2^2,\ |\nabla_{s_2}u_2(x)|_2^2=\theta^{\frac{s_1-s_2}{s_1}}|\nabla_{s_2}u_1(x)|_2^2,\ |u_2(x)|_2^2=a_2. \]
Recall that, from Remark \ref{rem}, we have
$$ A\beta-Bd=\frac{d}{2s_1}-\frac{d-2s_1}{4s_1}\beta>0.$$
Then, by Remark \ref{rem} and Lemma \ref{max}, we obtain
\begin{align*}
m(a_2)&\leq \max_{t>0} I(t*u_2)=I(t_{u_2}*u_2)=I(t_{u_2}^\frac{d}{2}\theta^A u_1(\theta^B t_{u_2}x))\\
&=\frac{1}{2}|\nabla_{s_1}(t_{u_2}*u_1)|_2^2+\frac{1}{2}\theta^{\frac{s_1-s_2}{s_1}}|\nabla_{s_2}(t_{u_2}*u_1)|_2^2 -\theta^{-B d}\int_{\R^d}G(t_{u_2}^\frac{d}{2}\theta^A u_1( t_{u_2}x))dx\\
&\leq \frac{1}{2}|\nabla_{s_1}(t_{u_2}*u_1)|_2^2+\frac{1}{2}|\nabla_{s_2}(t_{u_2}*u_1)|_2^2-\theta^{A\beta-Bd}\int_{\R^d}G(t_{u_2}^\frac{d}{2} u_1( t_{u_2}x))dx\\
&\leq \frac{1}{2}|\nabla_{s_1}(t_{u_2}*u_1)|_2^2+\frac{1}{2}|\nabla_{s_2}(t_{u_2}*u_1)|_2^2-\int_{\R^d}G(t_{u_2}*u_1 )dx\\
&=I(t_{u_2}*u_1)\leq \max\limits_{t>0}I(t*u_1),
\end{align*}

It follows, since $u_1\in S_{a_1}$ is arbitrary, that
\[m(a_2)\leq \inf_{u\in S_{a_1}}\max_{t>0} I(t*u)=m(a_1).\]
The equality holds only if there exists $u_{1,n}\in S_{a_1}$ such
that $\int_{\R^d} G(t_n*u_{1,n})dx\rightarrow 0$, where $t_n$
satisfies $t_n*u_{2,n}\in \mathcal{P}_{a_2}$. Then we have
\[\int_{\R^d} \widetilde{G}(t_n*u_{2,n})dx\rightarrow 0\ \text{as}\ n\rightarrow \infty.\]
Then by $t_n*u_{2,n}\in \mathcal{P}_{a_2}$, we have
\[|\nabla_{s_1} (t_n*u_{2,n})|_2\rightarrow 0,  \ |\nabla_{s_2} (t_n*u_{2,n})|_2\rightarrow 0, \ \text{as}\ n\rightarrow \infty.\]
Therefore, $I(t_n*u_{2,n})\rightarrow 0$. This contradict the fact
that $\inf_{u\in \mathcal{P}_{a_2} } I(u)>0$, which means $m$ is
strictly decreasing.
\end{proof}

\begin{Lem}\label{con}
Suppose that $(G_1)-(G_3)$ hold, then $m$ is continuous on $\R^+$.
\end{Lem}
\begin{proof}
Since $m$ is strictly decreasing (see Lemma \ref{dec}), for any
fixed $a>0$, $m(a-h)$ and $m(a+h)$ are monotonic and bounded as
$h\rightarrow 0+$, thus they have limits. Moreover, we know that
$$m(a-h)\geq m(a)\geq m(a+h).$$ Thus we have
\[\lim\limits_{h\rightarrow 0+} m(a-h)\geq m(a)\geq \lim\limits_{h\rightarrow 0_+}m(a+h).\]
Then we only need to show that $\lim\limits_{h\rightarrow 0+} m(a-h)\leq m(a)$ and $\lim\limits_{h\rightarrow 0_+}m(a+h)\geq m(a).$\\
Take $u\in \mathcal{P}_{\infty,a}$, for $h>0$, let
\[u_h(x):=\sqrt{1-\frac{h}{a}} u(x).\]
Evidently, $|u_h|_2^2=a-h$, and by Lemma \ref{max}, there exists
$t_h>0$, such that $t_h*u_h\in \mathcal{P}_{\infty,a-h}$.\\
\textbf{Claim:} $t_h$ is bounded in $\mathbb{R}$ and $t_h*u$ is
bounded in $H^{s_1,s_2}(\mathbb{R}^{d})$ as $h\rightarrow 0+$.\\
If $t_h<1$ the result follows directly. So, we can assume that
$t_h>1$. Now, since $t_h*u_h\in \mathcal{P}_{\infty,a-h}$, we get
\[s_1|\nabla_{s_1}(t_h*u_h)|_2^2+s_2|\nabla_{s_2}(t_h*u_h)|_2^2=d\int_{\R^d}\widetilde{G}(t_h*u_h)dx.\]
Therefore, we have
\begin{align*}
s_1t_h^{2s_1}|\nabla_{s_1}u_h|_2^2+s_2t_h^{2s_2}|\nabla_{s_2}
u_h|_2^2&=d\int_{\R^d}\widetilde{G}(t_h*u_h)dx
\geq d(\frac{\alpha}{2}-1)Ct_h^{\frac{\alpha-2}{2}d} |u_h|_\alpha^\alpha\\
&=Cd(\frac{\alpha}{2}-1) t_h^{\frac{\alpha-2}{2}d}
(1-\frac{h}{a})^\frac{\alpha}{2} |u|_\alpha^\alpha.
\end{align*}
It follows, for $h$ small enough, that
\begin{equation}\label{3.2.1}
s_1t_h^{2s_1} |\nabla_{s_1}u_h|_2^2+s_2t_h^{2s_2} |\nabla_{s_2}
u_h|_2^2  \geq
\frac{C}{2}d(\frac{\alpha}{2}-1)t_h^{\frac{\alpha-2}{2}d}
|u|_\alpha^\alpha.
\end{equation}
On the other hand, from the fact that  $\alpha>2+\frac{4s_2}{d}$
(see $(G_2)$), one has
\begin{equation}\label{3.2.2}
\frac{\alpha-2}{2}d>2s_2>2s_1.
\end{equation}
Thus, using \eqref{3.2.1} and \eqref{3.2.2}, $t_h$ is bounded
as $h\rightarrow 0+$, and so, $t_h*u$ is also bounded in $H^{s_1,s_2}(\R^d)$.  This proves the claim.\\
Similar to \cite{yang}, we obtain
 \begin{equation}\label{3.2.3}
 \int_{\R^d}(G(t_h*u_h)-G(t_h*u))dx=o(1),\ \text{as}\ \ h\rightarrow 0+.
\end{equation}
 Then, combining \eqref{3.2.3} with the claim, we obtain
\[I(t_h*u)-I(t_h*u_h)=\frac{h}{2a}(|\nabla_{s_1}(t_h*u)|_2^2+|\nabla_{s_2}(t_h*u)|_2^2)+\int_{\R^d} (G(t_h*u_h)-G(t_h*u))dx=o(1).\]
Thus,
\[m(a-h)\leq I(t_h*u_h)\leq I(t_h*u)+o(1)\leq I(u)+o(1)=m(a)+o(1).\]
This implies $$\lim\limits_{h\rightarrow 0+} m(a-h)\leq m(a).$$ Now,
take $u\in \mathcal{P}_{\infty,a}$ and
$v_h(x)=\frac{1}{\sqrt{1+\frac{h}{a}}}u(x)$. By the same argument
used above, we can deduce that
$$\lim\limits_{h\rightarrow 0_+}m(a+h)\geq m(a).$$
Therefore, $m$ is continuous on $\mathbb{R}^+$ and so we get the
proof of our desired result.
\end{proof}
\begin{Rek}
Similarly, we have $m_r:a\rightarrow m_{a,r}$ is continuous and strictly decreasing.
\end{Rek}
\begin{Lem}\label{coercive}
Let $(G_1)-(G_3)$ be satisfied. Then, the functional $I$ restricted
to $\mathcal{P}_{\infty,a}$ is coercive.
\end{Lem}
\begin{proof}
In light of $(G_2)$, we can deduce that for any $u\in
\mathcal{P}_{\infty,a}$
\begin{align}\label{Ico}
s_1|\nabla_{s_{1}}u|_{2}^{2}+s_2|\nabla_{s_{2}}u|_{2}^{2}&=d\int_{\R^{d}}\widetilde{G}(u)dx \leq \frac{d}{2}\beta \int_{\R^{d}}G(u)dx. 
\end{align}
It follows, using again $(G_2)$, that
\begin{align}\label{Icoer}
I(u)&= \frac{1}{2}|\nabla_{s_{1}}u|_{2}^{2}+
\frac{1}{2}|\nabla_{s_{2}}u|_{2}^{2}- \int_{\R^{d}}G(u)dx\\
&\geq\frac{1}{2s_2}\left[s_1|\nabla_{s_{1}}u|_{2}^{2}+
s_2|\nabla_{s_{2}}u|_{2}^{2}\right]- \int_{\R^{d}}G(u)dx\nonumber\\
&\geq
\frac{d}{4s_2}\int_{\R^{d}}\left(g(u)udx-(2+\frac{4s_2}{d})G(u)\right)dx\nonumber\\
&\geq \frac{d}{4s_2}(\alpha-2-\frac{4s_2}{d})\int_{\R^{d}}G(u)dx\nonumber\\
&\geq \frac{s_1}{2\beta} (\alpha-2-\frac{4s_2}{d})
\left[|\nabla_{s_{1}}u|_{2}^{2}+|\nabla_{s_{2}}u|_{2}^{2}\right],\nonumber
\end{align}
for any $u\in \mathcal{P}_{\infty,a}$. Here, we used the fact $\int_{\R^{d}} G(u)dx\geq0$ for any $u\in \mathcal{P}_{\infty,a}$.
This end the proof.
\end{proof}
\begin{Lem}\label{Ipositive}
Assume that the conditions of Lemma \ref{coercive} are fulfilled,
then
$$m_a>0.$$
\end{Lem}
\begin{proof}
Let $u\in \mathcal{P}_{\infty,a}$. Then, in view of $(G_2)$, we
have
\begin{align*}
s_{1}|\nabla_{s_{1}}u|_{2}^{2}&\leq
s_1|\nabla_{s_{1}}u|_{2}^{2}+s_2|\nabla_{s_{2}}u|_{2}^{2}\leq\frac{d}{2}\int_{\R^{d}}g(u)udx\leq
\frac{d\beta}{2}\int_{\R^d} G(u)dx\\
&\leq C\int_{\R^d}\left(|u|^{\alpha}+ |u|^{\beta} \right)dx.
\end{align*}
Here we used Remark \ref{rem} and the fact that $G(u)\geq 0$. To
estimate the right hand side, we apply \eqref{gagliardo} with
$p+2=\alpha$ and $p+2=\beta$, obtaining
\begin{equation}\label{positive1}
|\nabla_{s_{1}}u|_{2}^{2}\leq
C\left(|\nabla_{s_1}u|_{2}^{\frac{d(\alpha-2)}{2s_1}}+|\nabla_{s_1}u|_{2}^{\frac{d(\beta-2)}{2s_1}}
\right).
\end{equation}
Using the same argument, we get
\begin{equation}\label{positive2}
|\nabla_{s_{2}}u|_{2}^{2}\leq
C\left(|\nabla_{s_{2}}u|_{2}^{\frac{d(\alpha-2)}{2s_{2}}}+|\nabla_{s_{2}}u|_{2}^{\frac{d(\beta-2)}{2s_{2}}}\right).
\end{equation}
Hence, combining \eqref{positive1} and \eqref{positive2} and having
in mind that  $\frac{d(\alpha-2)}{2s_{1}}$,
$\frac{d(\alpha-2)}{2s_2}$, $\frac{d(\beta-2)}{2s_1}$ and
$\frac{d(\beta-2)}{2s_2}$ are strictly larger than $2$, we infer
that
\begin{equation}\label{positive3}
|\nabla_{s_{1}}u|_{2}^{2}+|\nabla_{s_{2}}u|_{2}^{2}\geq \delta,
\end{equation}
for some $\delta>0$. Therefore the result follows by combining
\eqref{Icoer} and \eqref{positive3}.
\end{proof}
\begin{Rek}
Actually, the functional $I$ restricted to $\mathcal{P}_{\infty,a,r}$ is also coercive, therefore, $m_{a,r}>0$.
\end{Rek}

\subsection{The radial case}
It should is clear that the lemmas established so far remain
unchanged if we work in the subspace
$H^{s_1,s_2}_{r}(\mathbb{R}^{d})$.
\begin{Thm}\label{dec}
Suppose that $(G_1)-(G_3)$ hold, then $m_{a,r}$\ is attained, i.e.
there exists $u\in H^{s_1,s_2}_{r}(\mathbb{R}^{d})$, such that
$$m_{a,r}=\displaystyle \inf_{v\in \mathcal{P}_{\infty,a,r}}I(v)=I(u),$$
where $\mathcal{P}_{\infty,a,r}= H^{s_1,s_2}_{r}(\R^{d})\bigcap \mathcal{P}_{\infty,a}.$\\
\end{Thm}

\begin{proof}
Let $\{u_n\}\subset \mathcal{P}_{\infty,a,r}$ be a minimizing
sequence of $m_{a,r}$, i.e.
\[I(u_n)\rightarrow m_{a,r},\ P_\infty(u_n)=0,\ |u_n|_2^2=a.\]
 Then, in view of Lemma \ref{coercive},   $\{u_n\}$ is bounded in
$H^{s_1,s_2}_{r}(\mathbb{R}^{d})$. Subtracting if necessary  a
subsequence, $u_n\rightharpoonup u$ in
$H^{s_1,s_2}(\mathbb{R}^{d})$. Thus, using $(G_1)-(G_2)$, we can
deduce that
$$\lim_{n \rightarrow +\infty} \int_{\R^{d}}G(u_n)dx=\int_{\R^{d}}G(u)dx,$$
and
$$\lim_{n \rightarrow +\infty} \int_{\R^{d}}g(u_n) u_ndx=\int_{\R^{d}}g(u)udx.$$
Therefore
\[I(u)\leq \liminf_{n\rightarrow \infty} I(u_n)=m_{a,r},\]
\[P_\infty(u)\leq\liminf_{n\rightarrow \infty}P_\infty(u_n)=0,\]
and
\[|u|_2^2\leq \liminf_{n\rightarrow \infty}|u_n|_2^2=a.\]
Now, from Lemma \ref{Ipositive} and the fact $P_{\infty}(u_n)=0$, we
can deduce that $u\neq 0.$ Then, let $\mu=|u|_{2}^{2}\in (0,a]$.
Observe that from the proof of Lemma \ref{max} and the fact that
$P_\infty(u)\leq 0$, there is $t_u\in (0,1]$ such that
$t_u*u\in \mathcal{P}_{\infty,\mu,r}$. Therefore, having in mind
  $\frac{1-t_{u}^{2s_2}}{2s_2}\leq \frac{1-t_{u}^{2s_1}}{2s_1}$,
  we obtain
\begin{align}\label{in1}
I(u)-I(t_u*u)&= \frac{1-t_{u}^{2s_1}}{2s_1}
s_1|\nabla_{s_1}u|_{2}^{2}+ \frac{1-t_{u}^{2s_2}}{2s_2}
s_2|\nabla_{s_2}u|_{2}^{2}+t_{u}^{-d}\int_{\R^d}
G(t_{u}^{\frac{d}{2}}u)dx-\int_{\R^d} G(u)dx\\
&\geq \frac{1-t_{u}^{2s_2}}{2s_2} \left[
s_1|\nabla_{s_1}u|_{2}^{2}+s_2|\nabla_{s_2}u|_{2}^{2}
\right]+t_{u}^{-d}\int_{\R^d} G(t_{u}^{\frac{d}{2}}u)dx-\int_{\R^d}
G(u)dx\nonumber\\
&= \frac{1-t_{u}^{2s_2}}{2s_2}
P_{\infty}(u)+d(\frac{1-t_{u}^{2s_2}}{4s_2})\int_{\R^d}g(u)udx-(d\frac{1-t_{u}^{2s_2}}{2s_2}+1)\int_{\R^d}G(u)dx\nonumber\\
&+t_{u}^{-d}\int_{\R^d}G(t_{u}^{\frac{d}{2}}u)dx. \nonumber
\end{align}
By $(G_3)$, we have
$$\frac{g(s)s-2G(s)}{|s|^{1+\frac{4s_2}{d}}s} \ \ \mbox{is increasing for} \ \ s\in \R\setminus \{0\},$$
and so
\begin{align*}
&d\left(\frac{1-t_{u}^{2s_2}}{4s_2}\right)\int_{\R^d}g(u)udx-\left(d\frac{1-t_{u}^{2s_2}}{2s_2}+1\right)\int_{\R^d}G(u)dx
+t_{u}^{-d}\int_{\R^d}G(t_{u}^{\frac{d}{2}}u)dx\\
&=\int_{\R^d}
\int_{t_u}^{1}\frac{d}{2}t^{2s_2-1}|u|^{2+\frac{4s_2}{d}}\left[
\frac{g(u)u-2G(u)}{|u|^{2+\frac{4s_2}{d}}}-\frac{g(t^{\frac{d}{2}}u)t^{\frac{d}{2}}u-2G(t^{\frac{d}{2}}u)}{|t^{\frac{d}{2}}u|^{2+\frac{4s_2}{d}}}\right]dt
dx\\
&\geq 0.
\end{align*}
Consequently,
$$I(u)-I(t_u*u)\geq \frac{1-t_{u}^{2s_2}}{2s_2}P_{\infty}(u),$$
which implies, using Lemma \ref{decc}, that
\begin{align*}
m_{\mu,r}\geq m_{a,r}&=\lim_{n\rightarrow \infty}
\Big(I(u_n)-\frac{1}{2s_2}P_\infty(u_n)\Big)\\
&\geq
\liminf_{n\rightarrow \infty}\left[ \left(\frac{s_2-s_1}{2s_2}\right)|\nabla_{s_{1}}u_n|_{2}^{2}+\frac{d}{2s_2}\int_{\R^d} \widetilde{G}(u_n)dx-\int_{\R^d}G(u_n)dx\right]\\
&\geq\left(\frac{s_2-s_1}{2s_2}\right)|\nabla_{s_{1}}u|_{2}^{2}+\frac{d}{2s_2}\int_{\R^d} \widetilde{G}(u )dx-\int_{\R^d}G(u )dx\\
&=I(u)-\frac{1}{2s_2}P_\infty(u)\geq
I(t_u*u)-\frac{t_u^{2s_2}}{2s_2}P_\infty(u)\\
&\geq I(t_u*u) \geq m_{\mu,r}.
\end{align*}
This proves that $m_{\mu,r}=m_{a,r}$. Hence, by Lemma \ref{decc}, we
deduce that $\mu=a$, which means $t_u*u\in \mathcal{P}_{\infty,a,r}$
and $I(t_u*u)=m_{a,r}$. Thus $I$ attains its minimum at $u\in
\mathcal{P}_{\infty, a,r}$. This proves our desired result.
\end{proof}

\subsection{The nonradial case}

In this part, we will try to give a version of Theorem \ref{dec} in
the nonradial Sobolev space $H^{s_1,s_2}(\R^d)$. 
We divide this subsection into two parts. Due to the lack of
compactness, the main purpose of the first part is to classify the
$(PS)_c$-sequences for a suitable functional $I_{\lambda}$. In the second
part, we prove that $m_a:= \inf\limits_{u\in \mathcal{P}_{\infty,a}}
I(u)$ is attained .\\

The main result of this subsection reads as follow.
\begin{Thm}\label{attained}
Suppose that $(G_1)-(G_3)$ hold, 
then $$m_a:=
\inf\limits_{u\in \mathcal{P}_{\infty,a}} I(u),$$ is attained.
\end{Thm}
 \subsubsection{Classification of $(PS)_c$-sequences}

 Let $I_{\lambda}:H^{s_1,s_2}(\R^d)\rightarrow \R$ be defined by
 \begin{equation}\label{Ilambda}
 I_\lambda(u)=\frac{1}{2}|\nabla_{s_1}u|_{2}^{2}+\frac{1}{2}|\nabla_{s_2}u|_{2}^{2}+\frac{\lambda}{2}|u|_{2}^{2}- \int_{\R^d} G(u )dx,
 \end{equation}
 where $\lambda>0$.\\

  The main purpose of this part, is to classify $(PS)_c$-sequences for the functional $I_{\lambda}$. Here, we use some ideas coming from
 \cite{benci,bhakta}.

 \begin{Thm}\label{Palais}
Let $\{u_n\}\subset H^{s_1,s_2}(\R^{d})$ be a $(PS)_c$-sequence. Then, there
exists $k\in \mathbb{N}$, $k$ functions $u^{1},..., u^{k}$ in
$H^{s_1,s_2}(\R^{d})$ and a subsequence (still denoted $\{u_n\}$) such
that\\
$(1)$ $I_{\lambda}^{'}(u^{i})=0$, for $i=1,...,k.$\\
$(2)$ $|\nabla_{s_{1}}u_{n}|_{2}^{2}\rightarrow
|\nabla_{s_{1}}u^{1}|_{2}^{2}+...+|\nabla_{s_{1}}u^{k}|_{2}^{2}$ and
$|\nabla_{s_{2}}u_{n}|_{2}^{2}\rightarrow
|\nabla_{s_{2}}u^{1}|_{2}^{2}+...+|\nabla_{s_{2}}u^{k}|_{2}^{2}$.\\
$(3)$ $|u_{n}|_{2}^{2}\rightarrow
|u^{1}|_{2}^{2}+...+|u^{k}|_{2}^{2}$.\\
$(4)$ $I_{\lambda}(u_n) \rightarrow I_{\lambda}(u^{1})+...+I_{\lambda}(u^{k}).$\\
$(5)$ If $k=1$, then there exists a sequence $\{y_n\}\subset \R^{d}$
such that $u_n(x)-u^{1}(x-y_{n})\rightarrow 0$ strongly in
$H^{s_1,s_2}(\R^{d})$.
 \end{Thm}
 Theorem \ref{Palais} is also called Representation Lemma in some references. To prove Theorem \ref{Palais}, we first give some auxiliary lemmas.
 \begin{Lem}\label{lions}
Let $t>0$ and $2\leq q<2^{\ast}_{s_{2}}$. If $\{u_n\}$ is a bounded
sequence in $H^{s_{1},s_{2}}(\R^{d})$ and if
$$\displaystyle \sup_{y\in \R^{d}}\int_{B(y,t)}|u_n|^{q}dx \rightarrow 0 \ \ \mbox{as} \ \ n\rightarrow +\infty,$$
then $\{u_n\}\rightarrow 0$ in $L^{r}(\R^{d})$ for all $r\in
(2,2^{\ast}_{s_{2}})$. 
 \end{Lem}
 \begin{proof}
 It is easy to see that $\{u_n\}$ is bounded in $H^{s_{2}}(\R^{d})$,
 then the rest of the proof is similar to Lemma $2.1$ in
 \cite{bhakta}.
 \end{proof}
 \begin{Lem}
 Suppose that $(G_1)$-$(G_3)$ hold, then $g:H^{s_1,s_2}\rightarrow (H^{s_1,s_2})^*, u\rightarrow g(u)$ is a compact operator.
 \end{Lem}
 \begin{proof}
The proof follows by applying the same argument used in the proof of
Lemma $2.6$ in \cite{jeanjean}.
 \end{proof}
\begin{Lem}\label{arret}
Let $u\in H^{s_1,s_2}(\R^{d})\setminus \{0\}$ be a critical point of $I_{\lambda}$.
Then, there exists a positive constant $M>0$ independent of $u$ such
that
$$\|u\|\geq M.$$
\end{Lem}
\begin{proof}
Fix $\lambda>0$, by the Sobolev embedding theorem, we can put
$$M'_p=\inf\{|\nabla_{s_{1}}v|_{2}^{2}+|\nabla_{s_{2}}v|_{2}^{2}+\lambda \int_{\R^{d}}|v(x)|^{2}dx: v\in H^{s_1,s_2}(\R^{d}), \ \ \int_{\R^{d}}|v(x)|^{p}dx=1 \}.$$
We can see that $M'_p>0$. Then
\begin{equation}\label{arret1}
|\nabla_{s_{1}}u|_{2}^{2}+|\nabla_{s_{2}}u|_{2}^{2}+\lambda
\int_{\R^{d}}|u(x)|^{2}dx\geq M'_\alpha|u|_\alpha^2.
\end{equation}
\begin{equation}\label{arret2}
|\nabla_{s_{1}}u|_{2}^{2}+|\nabla_{s_{2}}u|_{2}^{2}+\lambda
\int_{\R^{d}}|u(x)|^{2}dx\geq M'_\beta|u|_\beta^2.
\end{equation}
On the other hand, if $u$ is a critical point of $I_\lambda$, we have
\begin{equation}\label{arret3}
|\nabla_{s_{1}}u|_{2}^{2}+|\nabla_{s_{2}}u|_{2}^{2}+\lambda
\int_{\R^{d}}|u(x)|^{2}dx=\int_{\R^{d}}g(u)udx\leq \beta\int_{\R^{d}}G(u)dx\leq C(|u|_\alpha^\alpha+|u|_\beta^\beta).
\end{equation}
Hence, combining \eqref{arret1}, \eqref{arret2} and \eqref{arret3}, we obtain
$$\|u\|^2\leq C\left( \frac{\|u\|^\alpha}{(M_\alpha')^\frac{\alpha}{2}}+ \frac{\|u\|^\beta}{(M_\beta')^\frac{\beta}{2}} \right).$$
Thus, there exists $M>0$, such that $\|u\|\geq M$.
\end{proof}
\\
\textbf{ Proof of
Theorem \ref{Palais}:}\\
\textbf{Claim $1$:} The $(PS)_c$-sequence $\{u_n\}$ is bounded in
$H^{s_{1},s_{2}}(\R^{d})$. We have:
\begin{align*}
c+o(1)\|u_n\|&\geq I_{\lambda}(u_n)-\frac{1}{\alpha}\langle I^{'}_{\lambda}(u_n),u_n\rangle\\
&=\left(\frac{1}{2}-\frac{1}{\alpha}\right)\left(|\nabla_{s_1} u|_2^2+|\nabla_{s_2} u|_2^2+\lambda|u|_2^2\right)+\int_{\R^d} \left(\frac{1}{\alpha}g(u)u-G(u)\right)dx\\
&\geq C\|u\|^2.
\end{align*}
Hence boundedness follows. Thus we may extract a subsequence $\{u_n\}$
such that
$$u_n\rightharpoonup u^0\ \ \mbox{in} \ \ H^{s_{1},s_{2}}(\R^{d}), $$
and $$u_n(x) \rightarrow u^{0}(x) \ \ \mbox{a.e. in} \ \ \R^{d}.$$
By standard arguments, $u$ is a critical point of $I_{\lambda}.$ Now
let
$$\Psi^{1}_{n}(x)= (u_n-u^{0})(x), \ \ x\in \R^{d}.$$
Then
\begin{equation}\label{pp1}
\Psi_{n}^{1} \rightharpoonup 0 \ \ \mbox{in} \ \ H^{s_{1},s_{2}}(\R^{d}).
\end{equation}
Furthermore, by applying the Br$\acute{\text{e}}$zis-Lieb lemma \cite{brezis}, we
get
\begin{equation}\label{pp2}
\|\Psi_{n}^{1}\|^{2}= \|u_{n} \|^{2}-\|u^{0}\|^{2}+o(1),
\end{equation}
and
\begin{equation}\label{p3}
\int_{\R^d}G(\Psi_{n}^{1})dx=\int_{\R^d}G(u_n)dx-\int_{\R^d}G(u^0)dx+o(1).
\end{equation}
Thus, by taking into account of \eqref{pp1}, \eqref{pp2} and
\eqref{p3}, we infer that
\begin{equation}\label{p4}
I_{\lambda}(\Psi_{n}^{1})=I_{\lambda}(u_n)-I_{\lambda}(u^{0})+o(1),
\end{equation}
and
\begin{equation}\label{p5}
I_{\lambda}'(\Psi_{n}^{1})=I_{\lambda}'(u_n)-I_{\lambda}'(u^0)+o(1).
\end{equation}
Now, using Lemma \ref{lions} we have, either $\Psi_{n}^{1}
\rightarrow 0 $ in $H^{s_1,s_2}(\R^{d})$, in that case the proof is
over or there exists $\alpha^{'}>0$, such that up to a subsequence
$$\displaystyle \sup_{y\in \R^{d}}\int_{B(y,1)}{|\Psi_{n}^{1}(x)|}^{2}dx>\alpha^{'}>0.$$
Therefore we can find a sequence $\{y_n\}\subset \R^{d}$ such that
\begin{equation}\label{p6}
\int_{B(y_n,1)} {|\Psi_{n}^{1}(x)|}^{2}dx\geq \alpha^{'}.
\end{equation}
It follows,  from the fact that $\Psi_{n}^{1}\rightharpoonup 0$ in
$H^{s_1,s_2}(\R^{d})$ and \eqref{p6}, that
$$|y_n|\rightarrow+\infty \ \ \mbox{as} \ \ n\rightarrow +\infty.$$
Let us now call $u^1$ the weak limit in $H^{s_1,s_2}(\R^{d})$ of the
sequence $\Psi_{n}^{1}(.+y_n)$.\\
\textbf{Claim $2$:} the weak limit $u^1\neq 0.$ Indeed, since
$H^{s_1,s_2}(B(0,1))\hookrightarrow L^{2}(B(0,1))$ compactly
embedded,  \eqref{p6} concludes the claim. Moreover, by
standard argument, $u^{1}$ is a critical point of $I_{\lambda}$.\\
Iterating this procedure, we obtain sequences
$\Psi_{n}^{j}=\Psi^{j-1}_{n}(x+y_n)-u^{j-1}$ $j\geq 2$ and sequences
of points $y_{n}^{j}$ such that $|y_{n}^{j}|\rightarrow +\infty$ as
$n\rightarrow +\infty$ and
$$\Psi_{n}^{j}(\cdot+y_{n}^{j})\rightharpoonup u^{j} \ \ \mbox{in} \ \ H^{s_1,s_2}(\R^{d}), $$
where each $u^j$ is a critical point of $I_{\lambda}$. Moreover, in
view of \eqref{pp2} and \eqref{p4}, we get
\begin{equation} \label{p7}
\|\Psi_{n}^{j}\|^{2}=\|\Psi_{n}^{j-1}\|^{2}-\|u^{j-1}\|^{2}+o(1)=\|u_n\|^{2}-\|u^0\|^{2}-\displaystyle
\sum_{i=1}^{j-1}\|u^{i}\|^{2}+o(1),
\end{equation}
and
\begin{equation}\label{p8}
I(\Psi_{n}^{j})=I(\Psi_{n}^{j-1})-I(u^{j-1})+o(1)=I(u_n)-I(u_0)-
\displaystyle \sum_{i=1}^{j-1} I(u^{i})+o(1).
\end{equation}
Thus  Lemma \ref{arret}, Claim $1$, and \eqref{p7} tell us that the
iteration must terminate at some index $k\geq 0$.\\
If \ \ $k=0$, \ \ put \ \ $u_{n}^{0}\equiv u_{n}$.\\
If \ \ $k>0,$ \ \ put \ \ $u_{n}^{k}\equiv\Psi^{k}_{n}(.+y_{n}^{k})$, $u_{n}^{i}\equiv\Psi_{n}^{i}(.+y_{n}^{i})-\displaystyle
\sum_{j=i+1}^{k}u_{n}^{j}(\cdot-y_{n}^{j})$, for $0<i<k$, and $$u_{n}^{0}\equiv u_{n}-\displaystyle
\sum_{j=1}^{k}u_{n}^{j}(\cdot-y_{n}^{j}).$$\\
Then, the sequences $\{u_{n}^{j}\}, \ \ 0\leq j\leq K$ are the desired
sequences.

\subsection{Proof of Theorem \ref{attained}}
The purpose of this part is to conclude the proof of Theorem
\ref{attained}. The main idea is to give a mountain pass
characterization of $m_a$ which will be helpful to prove our desired
result. Here, we use some ideas coming from \cite{jeanjean}. 
\\

We need the following definition to introduce our variational
procedure.\\

$\bullet$ $\mathcal{H}: H^{s_1,s_2}(\R^d)\times \R \rightarrow
H^{s_1,s_2}(\R^d)$ with
$$\mathcal{H}(u,t)(x)=e^{\frac{dt}{2}}u(e^{t}x).$$

It is clear that $\left|\mathcal{H}(u,t)\right|_2=|u|_2$, for any $t\in \mathbb{R}$. We start by proving some technical lemmas which will be useful in the sequel.
\begin{Lem}\label{mg1}
Assume that $(G_1)-(G_2)$ hold and let $u\in S(a)$ be arbitrary.
Then we have:\\
$(1)$
$|\nabla_{s_1}\mathcal{H}(u,t)|_{2}+|\nabla_{s_2}\mathcal{H}(u,t)|_{2}
\rightarrow 0$, and $I(\mathcal{H}(u,t))\rightarrow 0$  as $t
\rightarrow -\infty$.\\
$(2)$
$|\nabla_{s_1}\mathcal{H}(u,t)|_{2}+|\nabla_{s_2}\mathcal{H}(u,t)|_{2}
\rightarrow +\infty$, and $I(\mathcal{H}(u,t))\rightarrow -\infty$
as $t \rightarrow +\infty$.
\end{Lem}
\begin{proof}
By a straightforward calculation, it follows that
\begin{equation}\label{mg11}
\int_{\R^{d}} |\mathcal{H}(u,t)(x)|^{2}dx=a^{2} \ \ \mbox{and} \ \
|\nabla_{s_1}\mathcal{H}(u,t)|_{2}^{2}+|\nabla_{s_2}\mathcal{H}(u,t)|_{2}^{2}=e^{2ts_1}
|\nabla_{s_1}u|_{2}^{2}+e^{2ts_2}|\nabla_{s_2}u|_{2}^{2} .
\end{equation}
Therefore
$$|\nabla_{s_1}\mathcal{H}(u,t)|_{2}+|\nabla_{s_2}\mathcal{H}(u,t)|_{2}
\rightarrow 0 \ \mbox{as} \ \ t \rightarrow -\infty.$$
 Using again \eqref{mg11} and Remark \ref{rem}, we find that
\begin{align*}
|I(\mathcal{H}(u,t))|&=\left|\frac{1}{2}
|\nabla_{s_1}\mathcal{H}(u,t)|_{2}^{2}+\frac{1}{2}|\nabla_{s_2}\mathcal{H}(u,t)|_{2}^{2}-\int_{\R^{d}}G(\mathcal{H}(u,t)(x))dx\right|\\
&\leq \frac{e^{2ts_{1}}}{2}
|\nabla_{s_1}u|_{2}^{2}+\frac{e^{2ts_{2}}}{2}|\nabla_{s_2}u|_{2}^{2}+
e^{dt (\frac{\alpha-2}{2})} \int_{\R^{d}} G(u)dx.
\end{align*}
Thus $$I(\mathcal{H}(u,t)) \rightarrow 0 \ \ \mbox{ as} \ \
t\rightarrow -\infty,$$ showing $(1)$.\\
In order to show $(2)$, note that by \eqref{mg11},
$$|\nabla_{s_1}\mathcal{H}(u,t)|_{2}^{2}+|\nabla_{s_2}\mathcal{H}(u,t)|_{2}^{2} \rightarrow +\infty \ \ \mbox{as} \ \ t \rightarrow +\infty.$$
On the other hand,
$$I(\mathcal{H}(u,t))\leq \frac{e^{2ts_{1}}}{2}
|\nabla_{s_1}u|_{2}^{2}+\frac{e^{2ts_{2}}}{2}|\nabla_{s_2}u|_{2}^{2}-
e^{dt (\frac{\alpha-2}{2})} \int_{\R^{d}} G(u)dx.$$ Taking into
account the fact that $\alpha
>2+\frac{4s_2}{d}$, it shows that
$I(\mathcal{H}(u,t))\rightarrow -\infty \ \ \mbox{as} \ \
t \rightarrow +\infty.$
\end{proof}
\begin{Lem}\label{mg2}
Suppose that the assumptions of Lemma \ref{mg1} are fulfilled. Then,
there exists $K(a)>0$ small enough such that
$$\displaystyle 0<\sup_{u\in A}I(u)<\inf_{u\in B}I(u)$$
with
$$A=\{u\in S(a): |\nabla_{s_1}u|_{2}^{2}+|\nabla_{s_2}u|_{2}^{2}\leq K(a)\} \ \ \mbox{and} \ \ B=\{u\in S(a): |\nabla_{s_1}u|_{2}^{2}+|\nabla_{s_2}u|_{2}^{2}=2 K(a).\}$$
\end{Lem}
\begin{proof}
Recall  from Remark \ref{rem} that
$$\int_{\R^d}G(u)dx\leq C\left[ \int_{\R^d}|u|^{\alpha}dx+\int_{\R^d}|u|^{\beta}dx\right],$$
where $C$ is a positive constant. Then, by the Gagliardo-Nirenberg
inequality \eqref{gagliardo}, we have $$\int_{\R^{d}}G(u)dx\leq
C\left[ \left(
|\nabla_{s_{1}}u|^2_{2}+|\nabla_{s_{2}}u|^2_{2}\right)^{d\frac{(\alpha-2)}{4s_{2}}}+\left(
|\nabla_{s_{1}}u|^2_{2}+|\nabla_{s_{2}}u|^2_{2}\right)^{d\frac{(\beta-2)}{4s_{2}}}\right].$$
Since $G(u)\geq 0$ for any $u\in H^{s_{1},s_{2}}(\R^{d})$, we can
deduce that
\begin{align*}
I(v)-I(u)&=
\frac{1}{2}\left[|\nabla_{s_{1}}v|_{2}^{2}+|\nabla_{s_{2}}v|_{2}^{2}
\right]-\frac{1}{2}\left[|\nabla_{s_{1}}u|_{2}^{2}+|\nabla_{s_{2}}u|_{2}^{2}
\right]- \int_{\R^d}G(v)dx+\int_{\R^d}G(u)dx\\
&\geq
\frac{1}{2}\left[|\nabla_{s_{1}}v|_{2}^{2}+|\nabla_{s_{2}}v|_{2}^{2}
\right]-\frac{1}{2}\left[|\nabla_{s_{1}}u|_{2}^{2}+|\nabla_{s_{2}}u|_{2}^{2}
\right]- \int_{\R^d}G(v)dx, \ \ \forall u,v\in
H^{s_{1},s_{2}}(\R^{d}).
\end{align*}
Therefore, if we fix $u\in A$, $v\in B$, then
$
|\nabla_{s_{1}}u|_{2}^{2}+|\nabla_{s_{2}}u|_{2}^{2}\leq K(a)$ and
$|\nabla_{s_{1}}v|_{2}^{2}+|\nabla_{s_{2}}v|_{2}^{2}=2K(a)$,
$$I(v)-I(u)\geq \frac{K(a)}{2}-C\left[(K(a))^{d\frac{\alpha-2}{4s_{2}}}+(K(a))^{d\frac{\beta-2}{4s_{2}}}\right],$$
where $C=C(\alpha,\beta, d,s_2,a)>0$. Due to $\beta>\alpha> 2+\frac{4s_2}{d}$, for $K(a)$ sufficiently small, there exists $\delta>0$, such that
$$I(v)-I(u)\geq \delta>0.$$
Moreover, exploiting the above
inequalities, we deduce that
$$\sup_{u\in A}I(u) >0, $$
where $K(a)$ small enough.
\end{proof}

As a consequence of the above two lemmas we get:
\begin{Lem}
Let $(G_1)-(G_2)$ be satisfied. Then there exist $u_1,u_2\in S(a)$
such that\\
$(1)$ $|\nabla_{s_{1}}u_1|_{2}^{2}+|\nabla_{s_{2}}u_1|_{2}^{2}\leq
K(a)$ ($K(a)$ is defined in Lemma \ref{mg2}).\\
$(2)$ $|\nabla_{s_{1}}u_2|_{2}^{2}+|\nabla_{s_{2}}u_2|_{2}^{2}>2 K(a)$.\\
$(3)$ $I(u_1)>0>I(u_2)$.\\
Moreover setting
$$\gamma(a)=\displaystyle \inf_{h\in \Gamma(a)}\max_{t\in[0,1]}I(h(t)),$$
with
$$\Gamma(a)=\{h\in C([0,1], S(a)), \ \ h(0)=u_1, \ \ h(1)=u_2\}.$$
Then
$$\gamma(a)>\max\{I(u_1),I(u_2)\}>0.$$
\end{Lem}
\begin{proof}
The proof follows directly from Lemmas \ref{mg1} and \ref{mg2}.
\end{proof}

To get our result  we shall draw an additional variational
characterization of the mountain pass level $\gamma(a)$.
\begin{Lem}\label{caractgam}
If $u\in H^{s_1,s_2}(\mathbb{R}^{d})$ is a critical point of $I$,
then
\[P_{\infty}(u)=0.\]
\end{Lem}
\begin{proof}
Let $u$ be a critical point of $I$. Thus, inserting  $u$ and $x\cdot
\nabla u$ respectively in the definition of $I^{'}$, and then
integrating by parts, we get
 \begin{equation}\label{p11}
|\nabla_{s_{1}}u|_{2}^{2}+|\nabla_{s_{2}}u|_{2}^{2}+\lambda
|u|_{2}^{2}-\int_{\mathbb{R}^{d}}g(u)udx=0
 \end{equation}
 and
 \begin{equation}\label{p12}
\frac{2s_1-d}{2}|\nabla_{s_{1}}u|_{2}^{2}+\frac{2s_2-d}{2}|\nabla_{s_{2}}u|_{2}^{2}-\frac{d\lambda}{2}|u|_{2}^{2}+d\int_{\mathbb{R}^{d}}G(u)dx=0.
 \end{equation}
 Exploiting \eqref{p11} and \eqref{p12}, we infer that
 $$s_1|\nabla_{s_{1}}u|_{2}^{2}+s_2|\nabla_{s_{2}}u|_{2}^{2}-d\int_{\mathbb{R}^{d}}
[\frac{g(u)u}{2}-G(u)]dx=0,$$
 which proves $P_{\infty}(u)=0,$ thus the proof of the lemma.
\end{proof}
\begin{Lem}\label{uniqueness}
Assume that $(G_1)-(G_3)$ hold and let $u\in S(a)$ be arbitrary.
Then, there exists a unique $t_u>0$ such that
$\mathcal{H}(u,t_u)\in \mathcal{P}_{\infty,a}$. Moreover,
$$I(\mathcal{H}(u,t_u))=\displaystyle \max_{t>0} I (\mathcal{H}(u,t)).$$
\end{Lem}
\begin{proof}
Fix $u\in S(a)$ and consider the function $f_u: \R \rightarrow \R$
defined by
$$f_u(t)=I(\mathcal{H}(u,t)).$$
Therefore
\begin{equation}\label{un1}
f_{u}'(t)=s_1|\nabla_{s_{1}} \mathcal{H}(u,t|_{2}^{2}+
s_2|\nabla_{s_{2}}
\mathcal{H}(u,t)|_{2}^{2}-d\int_{\R^{d}}\widetilde{G}(\mathcal{H}(u,t)(x))dx.
\end{equation}
In light of Lemmas \ref{mg1} and \ref{mg2}, there is $t_0\in \R$
such that $f_{u}^{'}(t_0)=0$. Thus, by \eqref{un1} we deduce  that
$\mathcal{H}(u,t_0)\in \mathcal{P}_{\infty,a}.$
Now, by direct computation and using \eqref{un1}, we get
\begin{align}\label{un2}
f''_{u}(t_0)&=2s_{1}^{2}|\nabla_{s_{1}}
\mathcal{H}(u,t_0)|_{2}^{2}+2s_{2}^{2}|\nabla_{s_{2}}
\mathcal{H}(u,t_0)|_{2}^{2}+ d^{2} \int_{\R^{d}}\widetilde{G}(\mathcal{H}(u,t_0)(x))dx\nonumber\\
&-\frac{d^{2}}{2}\int_{\R^{d}}\widetilde{G}^{'}(\mathcal{H}(u,t_0)(x))\mathcal{H}(u,t_0)(x)dx\nonumber\\
&\leq 2s_{2} \left( s_{1}|\nabla_{s_{1}}
\mathcal{H}(u,t_0)|_{2}^{2}+s_{2}|\nabla_{s_{2}}
\mathcal{H}(u,t_0)|_{2}^{2}\right)+d^{2}\int_{\R^{d}}\widetilde{G}(\mathcal{H}(u,t_0)(x))dx \nonumber\\
&-\frac{d^{2}}{2}\int_{\R^{d}}\widetilde{G}^{'}(\mathcal{H}(u,t_0)(x))\mathcal{H}(u,t_0)(x)dx\nonumber\\
&=(2s_2d+d^2)\int_{\R^{d}}\widetilde{G}(\mathcal{H}(u,t_0)(x))dx-\frac{d^2}{2}\int_{\R^{d}}\widetilde{G}^{'}(\mathcal{H}(u,t_0)(x))\mathcal{H}(u,t_0)(x)dx \nonumber\\
&\leq (2s_2d+d^2-\frac{\alpha d^2}{2})\int_{\R^{d}}\widetilde{G}(\mathcal{H}(u,t_0)(x))dx<0.
\end{align}
Hence, from $(G_3)$, $f''_{u}(t_0)<0$ and this proves the unicity
of $t_0$.
\end{proof}
\begin{Lem}\label{gammainf}
Assume that $(G_1)-(G_3)$ hold. Then
$$\gamma(a)=\displaystyle \inf_{u\in \mathcal{P}_{\infty,a}}I(u).$$
\end{Lem}
\begin{proof}
We argue by contradiction. We suppose that there is $v\in
\mathcal{P}_{\infty,a}$ such that
\begin{equation}\label{inf1}
I(v)<\gamma(a).
\end{equation}
From Lemma \ref{mg1}, there exists $t_0>0$ such that
$\mathcal{H}(v,-t_0)\in A$ ($A$ defined in lemma \ref{mg2}),
$|\nabla_{s_{1}}\mathcal{H}(v,t_0)|_{2}^{2}+
|\nabla_{s_{2}}\mathcal{H}(v,t_0)|_{2}^{2}\geq 2 K(a)$ and
$I(\mathcal{H}(v,t_0) )<0$. Now define the path $h:[0,1] \rightarrow
S(a)$
$$h(t)=\mathcal{H}(v,(2t-1)t_0).$$
Clearly $h(0)=\mathcal{H}(v,-t_0)$ and $h(1)=\mathcal{H}(v,t_0)$.
Then, by Lemma \ref{uniqueness}, we know that
$$\gamma(a)\leq \displaystyle \max_{t\in [0,1]}I(h(t))=I(v),$$
which reach to a contradiction with \eqref{inf1}.\\
This concludes the proof.
\end{proof}

Finally, in order to deduce the proof of our main theorem, we shall
need the following result.
\begin{Lem}\label{somme}
Assume that $(G_1)-(G_3)$ hold. Let $k\in \N$ and $a_1,a_2,\cdots,a_k$
be such that $a^2=a_{1}^{2}+\cdots+a_{k}^{2}$. Then
$$\gamma(a)< \gamma(a_1)+\cdots+\gamma(a_k).$$
\end{Lem}
\begin{proof}
The proof is similar to Lemma $2.11$ in \cite{jeanjean}.
\end{proof}

Now, we want to prove that there exists a $(PS)_c$-sequence for $I$
restricted to $S(a)$ at the level $c=\gamma(a)$. To achieve this
goal, we give the following two Lemmas inspired by Lemmas $2.4$ and
$2.5$ from the work of L. Jeanjean \cite{jeanjean}. This type of
proof has now become classical and we won't provide the details in
the sequel.
\begin{Lem}\label{PSF}
Assume that $(G_1)-(G_2)$ hold. Then there exists a sequence
$\{v_n\}\subset S(a)$ such that\\
$(1)$ $I(v_n) \rightarrow \gamma(a)$.\\
$(2)$ $\{\|v_n\|\}$ and $\left\{\int_{\R^{d}}G(v_n(x))dx\right\}$ are bounded in $\R$.\\
$(3)$ $|\langle
I'(v_n),z\rangle_{(H^{s_1,s_2}(\R^{d}))^{\ast}}|\leq
\frac{4}{\sqrt{n}},$ for all $z\in T_{v_{n}}$, where
$$T_{v_{n}}=\left\{z\in H^{s_1,s_2}(\R^{d}), \ \ \int_{\R^{d}}zv_n dx=0\right\}.$$
\end{Lem}
Next, we give a characterization of the sequence $\{v_n\}$ obtained in
Lemma \ref{PSF}.
\begin{Lem}\label{CPSF}
Assume that $(G_1)-(G_2)$ hold and let $\{v_n\}\subset
H^{s_1,s_2}(\R^d)$ be the sequence obtained in Lemma \ref{PSF}.
Then, up to a subsequence, we have:\\
$(1)$ $v_n \rightharpoonup v_a$ in $H^{s_1,s_2}(\R^{d})$.\\
$(2)$ $\int_{\R^d}G(v_n)dx \rightarrow c$ and
$\int_{\R^d}\widetilde{G}(v_n)dx \rightarrow d$ with
$c ,d>0$.\\
$(3)$ $\lambda_{n}=\frac{1}{|v_n|_{2}^{2}}\left(
|\nabla_{s_1}v_n|_{2}^{2}+|\nabla_{s_2}v_n|_{2}^{2}-\int_{\R^d}g(v_n)v_n
dx\right) \rightarrow \lambda_a<0$ in $\R$.\\
$(4)$ $(-\Delta)^{s_1}v_n+(-\Delta)^{s_2}v_n-\lambda_nv_n-g(v_n) \rightarrow 0$
in $(H^{s_1,s_2}(\R^{d}))^{\ast}.$
\end{Lem}

Consequently, we prove that the sequence $\{v_{n}\}$ obtained in Lemma
\ref{PSF} is a $(PS)_c$-sequence of $I_{-\lambda_{a}}$ (see
\eqref{Ilambda} ) where $\lambda_{a}$ is given in Lemma
 \ref{CPSF}.
 \begin{Lem}\label{PS}
Assume that $(G_1)-(G_3)$ hold. Then, the sequence $\{v_n\}$ is a $(PS)_d$-sequence of $I_{-\lambda_{a}}$.
 \end{Lem}
\begin{proof}
Using the fact that
$$\lambda_{n}=\frac{1}{|v_n|_{2}^{2}}\left(
|\nabla_{s_1}v_n|_{2}^{2}+|\nabla_{s_2}v_n|_{2}^{2}-\int_{\R^d}g(v_n)v_n
dx\right) \rightarrow \lambda_a
$$
and the boundedness of $\{v_n\}$ (see Lemma \ref{CPSF}), we can deduce
that
$$I_{-\lambda_{a}}(v_n)- \int_{\R^{d}}\widetilde{G}(v_n)dx\rightarrow 0.$$
Thus, again by Lemma \ref{CPSF}, $I_{-\lambda_{a}} \rightarrow d>0$.
Moreover, since we have
$$(-\Delta)^{s_1}v_n+(-\Delta)^{s_2}v_n-\lambda_av_n-g(v_n) \rightarrow 0 \ \
\mbox{in} \ \ (H^{s_1,s_2}(\R^{d}))^{\ast},$$ then $\{v_n\}$ is a $(PS)_d$-sequence of $I_{-\lambda_{a}}$.\\
This completes the proof.
\end{proof}

\textbf{Proof of Theorem \ref{attained} concluded.} By Lemma
\ref{PS} and Theorem \ref{Palais}, there exist $v^1,v^2,..., v^k$
such that
$$|v_n|_{2}^{2}\rightarrow \displaystyle \sum_{i=1}^{k} |v^i|_{2}^{2}, \ \
 |\nabla_{s_1}v_n|_{2}^{2}+|\nabla_{s_2}v_n|_{2}^{2} \rightarrow \displaystyle \sum_{i=1}^{k}|\nabla_{s_1}v^i|_{2}^{2}+|\nabla_{s_2}v^i|_{2}^{2}, \ \
 I_{-\lambda_{a}}(v_n) \displaystyle\rightarrow \sum_{i=1}^{k}I_{-\lambda_{a}}(v^i)$$
 which implies that $$a^2=\sum_{i=1}^{k} |v^i|_{2}^{2}, \ \ I(v_n)\rightarrow  \displaystyle \sum_{i=1}^{k}I(v^i) \ \ \mbox{and} \ \ \gamma(a)=\displaystyle \sum_{i=1}^{k}I(v^i).$$
\textbf{ Claim.} We shall prove that $k=1$. Otherwise, we assume
that $k\geq
 2$. Then, from Theorem \ref{Palais} and Lemma \ref{gammainf}, we know that
 $$\gamma(a)=\displaystyle \inf_{u\in \mathcal{P}_{\infty,a}}I(u),$$
 and so
 $$\gamma(a)=\displaystyle \sum_{i=1}^{k}I(v^i) \geq \displaystyle \sum_{i=1}^{k} \gamma(a_i).$$
This reach to a contradiction with Lemma \ref{somme} and so the
claim is proved. \\
Now the function $v^1$ satisfies
$$I(v^1)=\gamma(a), \ \ |v^1|^2_{2}=a,$$
and
$$(-\Delta)^{s_1}v^1+(-\Delta)^{s_2}v^1-g(v^1)=\lambda_a v^1.$$
This ends the proof of Theorem \ref{attained}.
\section{Proof of Theorem \ref{mainnthm}}

Our aim in this section is the proof of Theorem \ref{mainnthm}. This
section is divided into two subsections. In the first one, we give
some technical lemmas which will be used later. In the second
subsection, we conclude the proof of our main result.

\subsection{Technical Lemmas}
In this part, we give some technical lemmas related to problem
\eqref{maineq} with general nonlinear term since we don't need to use Theorem \ref{Palais}.
 First, we can remark that:
 $$P(u)=\Psi_u'(t)|_{t=1}.$$ Then, we
have:
\begin{Pro}\label{PP}
$\Psi_u'(t)=\frac{1}{t}P(t*u)$, for all $u\in H^{s_1,s_2}(\R^d)$.
\end{Pro}
\begin{proof}
First, we claim that $|\nabla_s(t*u)|_2^2=t^{2s}|\nabla_s u|_2^2$,
for all $s\in (0,1)$:
\begin{align*}
|\nabla_s(t^{\frac{d}{2}}u(tx))|_2^2&=\int_{\mathbb{R}^d\times \mathbb{R}^d}\frac{t^d|u(tx)-u(ty)|^2}{|x-y|^{d+2s}}dxdy\\
&=t^{2s}\int_{\mathbb{R}^d\times \mathbb{R}^d}\frac{|u(tx)-u(ty)|^2}{|tx-ty|^{d+2s}}d(tx)d(ty)\\
&=t^{2s}|\nabla_su|_2^2.
\end{align*}
Thus,
\[\Psi_u(t)=\frac{t^{2s_1}}{2}|\nabla_{s_1}u|_2^2+\frac{t^{2s_2}}{2}|\nabla_{s_2}u|_2^2+\frac{1}{2}\int_{\R^d}V(\frac{x}{t})u^2dx
-t^{-d}\int_{\R^d}G(t^{\frac{d}{2}}u(x))dx.\] Therefore,
\begin{align*}
\Psi_u'(t)&=s_1 t^{2s_1-1} |\nabla_{s_1}u|_2^2+s_2 t^{2s_2-1}
|\nabla_{s_2}u|_2^2-\frac{1}{2t^2}\int_{\R^d}\langle \nabla
V(\frac{x}{t}),x\rangle u^2dx -\frac{d}{t^{d+1}}\displaystyle
\int_{\mathbb{R}^{d}}\widetilde{G}(t^{\frac{d}{2}}u(x))dx\\
&=\frac{1}{t}P(t*u).
\end{align*}
This ends the proof.
\end{proof}
\begin{Lem}\label{P}
If $u\in H^{s_1,s_2}(\mathbb{R}^{d})$ is a solution of equation
\eqref{maineq}, then
\[P(u)=0.\]
\end{Lem}
\begin{proof}
Let $u$ be a solution of \eqref{maineq}. Thus, multiplying
\eqref{maineq} by $u$ and $x\cdot \nabla u$ respectively, and then
integrating by parts, we get
 \begin{equation}\label{p1}
|\nabla_{s_{1}}u|_{2}^{2}+|\nabla_{s_{2}}u|_{2}^{2}+\displaystyle
\int_{\mathbb{R}^{d}}V(x)u^{2}dx+\lambda
|u|_{2}^{2}-\int_{\mathbb{R}^{d}}g(u)udx=0
 \end{equation}
 and
 \begin{equation}\label{p2}
\frac{2s_1-d}{2}|\nabla_{s_{1}}u|_{2}^{2}+\frac{2s_2-d}{2}|\nabla_{s_{2}}u|_{2}^{2}-\frac{d}{2}\int_{\mathbb{R}^{d}}V(x)u^{2}dx-\int_{\mathbb{R}^{d}}\frac{\langle\nabla
V(x),x\rangle}{2}
u^{2}dx-\frac{d\lambda}{2}|u|_{2}^{2}+d\int_{\mathbb{R}^{d}}G(u)dx=0.
 \end{equation}
 From \eqref{p1} and \eqref{p2}, we obtain
 $$s_1|\nabla_{s_{1}}u|_{2}^{2}+s_2|\nabla_{s_{2}}u|_{2}^{2}-\int_{\mathbb{R}^{d}}\frac{\langle\nabla
V(x),x\rangle}{2} u^{2}dx-d\int_{\mathbb{R}^{d}}
[\frac{g(u)u}{2}-G(u)]dx=0,$$
 which is exactly $P(u)=0.$
\end{proof}

\begin{Pro}\label{psi}
Let $u\in S_a$. Then $t\in \R^+$ is a critical point for $\Psi_u(t)$ if and only if $t*u\in \mathcal{P}_a$.
\end{Pro}
\begin{proof}
The proof follows directly by using the fact that
$$\Psi'_{u}(t)=\frac{1}{t}P[t*u], \ \ \mbox{for all} \ \ t\in \R^+ \ \ \mbox{and} \ \ u\in S_a.$$
\end{proof}

\begin{Pro}\label{p4.3}
For any critical point of $J |_{\mathcal{P}_a}$, if $(\Psi_u)''(1)\neq 0$, then $\exists \lambda \in \R$ such that
\[J'(u)+\lambda u=0.\]
\end{Pro}
\begin{proof}
Let $u\in H^{s_1,s_2}(\mathbb{R}^{d})$ be a critical point of $J$
constraint on $\mathcal{P}_a$, then there are $\lambda,\mu\in
\mathbb{R}$ such that
$$J'(u)+\lambda u+\mu
P'(u)=0.$$ It is sufficient to prove that $\mu=0$. It is easy to
see that $u$ satisfies the following identity
$$\frac{\partial}{\partial t} \Phi (t*u)|_{t=1}=0, $$
where
\begin{align*}
\Phi(t*u)&=J(t*u)+\frac{\lambda}{2}|t*u|_{2}^{2}+\mu P(t*u)\\
&=\Psi_{u}(t)+\frac{\lambda}{2}|u|_{2}^{2}+\mu t \Psi^{'}_{u}(t).
\end{align*}
In the last equality we used Proposition \ref{PP}. Then
\begin{align*}
\frac{\partial}{\partial t} \Phi (t*u)|_{t=1}&=\Psi'_{u}(1)+\mu
\Psi''_{u}(1)+\mu \Psi^{'}_{u}(1)\\
&=(1+\mu)\Psi'_{u}(1)+\mu \Psi''_{u}(1)\\
&=(1+\mu)P(u)+\mu \Psi''_{u}(1)\\
&=\mu \Psi''_{u}(1).
\end{align*}
It follows, since $\Psi''_{u}(1)\neq0,$ that $\mu=0.$
\end{proof}

\begin{Lem}\label{lem4.2}
Suppose that  $(G_1)$-$(G_2)$ and $(V_2)$ are satisfied, then for
any $a>0$, $\exists \delta_a>0$ such that
\[\inf\{t>0: \exists u\in S_a, |\nabla_{s_{1}} u|_2^2+|\nabla_{s_{2}} u|_2^2=1, \ \text{such\ that\ } t*u\in \mathcal{P}_a\}\geq \delta_a.\]
Consequently,
$$\displaystyle\inf_{u\in \mathcal{P}_a}(|\nabla_{s_1}u|_2^2+|\nabla_{s_2}u|_2^2)\geq
\delta_a^2>0.$$
\end{Lem}
\begin{proof}
We recall that
\begin{equation}\label{m0}
\Psi'_{u}(t)=s_1t^{2s_1-1}|\nabla_{s_1}u|_{2}^{2}+s_2t^{2s_2-1}|\nabla_{s_2}u|_{2}^{2}-\frac{1}{t}\displaystyle
\int_{\mathbb{R}^{d}}W(\frac{x}{t})u^{2}dx
-\frac{d}{t^{d+1}}\displaystyle
\int_{\mathbb{R}^{d}}\widetilde{G}(t^{\frac{d}{2}}u(x))dx.
\end{equation} In light of condition $(V_2)$, we infer that
\begin{align}\label{m1}
&s_1t^{2s_1-1}|\nabla_{s_1}u|_{2}^{2}+s_2t^{2s_2-1}|\nabla_{s_2}u|_{2}^{2}-\frac{1}{t}\displaystyle
\int_{\mathbb{R}^{d}}W(\frac{x}{t})u^{2}dx\\
&=s_1t^{2s_1-1}|\nabla_{s_1}u|_{2}^{2}+s_2t^{2s_2-1}|\nabla_{s_2}u|_{2}^{2}-\frac{1}{t}\displaystyle
\int_{\mathbb{R}^{d}} W(x)(t*u)^{2}dx\nonumber\\
&\geq
s_1t^{2s_1-1}|\nabla_{s_1}u|_{2}^{2}+s_2t^{2s_2-1}|\nabla_{s_2}u|_{2}^{2}-\frac{\sigma_2}{t}
[|\nabla_{s_1}(t*u)|_{2}^{2}+|\nabla_{s_1}(t*u)|_{2}^{2}]\nonumber\\
&\geq
s_1t^{2s_1-1}|\nabla_{s_1}u|_{2}^{2}+s_2t^{2s_2-1}|\nabla_{s_2}u|_{2}^{2}-\sigma_2
[t^{2s_1-1}|\nabla_{s_1}(u)|_{2}^{2}+t^{2s_2-1}|\nabla_{s_2}(u)|_{2}^{2}]\nonumber\\
&\geq
(s_1-\sigma_2)[t^{2s_1-1}|\nabla_{s_1}(u)|_{2}^{2}+t^{2s_2-1}|\nabla_{s_2}(u)|_{2}^{2}]\nonumber.
\end{align}
Let $u\in S_a$ with $ |\nabla_{s_{1}} u|_2^2+|\nabla_{s_{2}}
u|_2^2=1$ and $t>0$ such that $t*u\in\mathcal{P}_a$. Then, invoking
Proposition \ref{psi}  and exploiting \eqref{m0} and \eqref{m1}, we
deduce that
\begin{align*}
(s_1-\sigma_2)\min(t^{2s_1-1},t^{2s_2-1}) &\leq
s_1t^{2s_1-1}|\nabla_{s_1}u|_{2}^{2}+s_2t^{2s_2-1}|\nabla_{s_2}u|_{2}^{2}-\frac{1}{t}\displaystyle
\int_{\mathbb{R}^{d}}W(\frac{x}{t})u^{2}dx\\
&=\frac{d}{t^{d+1}}\displaystyle
\int_{\mathbb{R}^{d}}\widetilde{G}(t^{\frac{d}{2}}u(x))dx.
\end{align*}
That is
\begin{equation}\label{m2}
(s_1-\sigma_2)\leq\frac{d\int_{\mathbb{R}^{d}}g(t^{\frac{d}{2}}u)udx}{2t^{\frac{d}{2}+1}(\min(t^{2s_1-1},t^{2s_2-1}))}
-\frac{d\int_{\mathbb{R}^{d}}G(t^{\frac{d}{2}}u)dx}{t^{d+1}(\min(t^{2s_1-1},t^{2s_2-1}))}
\end{equation}
Now, to conclude our proof, we distinguish two cases: \\
\textbf{Case $1$:} If $t>1$. Then, from \eqref{m2}, we get
\begin{equation}\label{m3}
(s_1-\sigma_2)\leq\frac{d\int_{\mathbb{R}^{d}}g(t^{\frac{d}{2}}u)udx}{2t^{\frac{d}{2}+2s_1}}
-\frac{d\int_{\mathbb{R}^{d}}G(t^{\frac{d}{2}}u)dx}{t^{d+2s_1}}.
\end{equation}
This implies, using condition $(G_2)$, that
\begin{equation}\label{m4}
(s_1-\sigma_2)\leq\frac{d\int_{\mathbb{R}^{d}}g(t^{\frac{d}{2}}u)udx}{2t^{\frac{d}{2}+2s_1}}
-\frac{d\int_{\mathbb{R}^{d}}g(t^{\frac{d}{2}}u)udx}{\beta
t^{\frac{d}{2}+2s_1}}=dt^{-d-2s_1}(\frac{1}{2}-\frac{1}{\beta})\int_{\mathbb{R}^{d}}g(t^{\frac{d}{2}}u)t^{\frac{d}{2}}udx.
\end{equation}
Again by $(G_2)$ and the fact that $ |\nabla_{s_{1}}
u|_2^2+|\nabla_{s_{1}} u|_2^2=1$ and $|u|_{2}^{2}=a$, there is $C>0$
such that
\begin{equation}\label{m5}
\int_{\mathbb{R}^{d}}g(t^{\frac{d}{2}}u)t^{\frac{d}{2}}udx\leq
C[t^{\frac{d}{2}\alpha}+t^{\frac{d}{2}\alpha}].
\end{equation}
Consequently, by combining \eqref{m4} and \eqref{m5}, we infer that
$$(s_1-\sigma_2)\leq Cd(\frac{1}{2}-\frac{1}{\beta})[t^{\frac{d}{2}\alpha-d-2s_1}+t^{\frac{d}{2}\beta-d-2s_1}].$$
So by, $2+\frac{4s_1}{d}<\alpha<\beta$, we obtain the lower bound of
$\delta_a$.\\
 \textbf{Case $2$:} If $t<1$. By the same argument used in case $1$,
we obtain
$$
(s_1-\sigma_2)\leq
Cd(\frac{1}{2}-\frac{1}{\beta})[t^{\frac{d}{2}\alpha-d-2s_2}+t^{\frac{d}{2}\beta-d-2s_2}].
$$
So by, $2+\frac{4s_2}{d}<\alpha<\beta$, we obtain the lower bound of
$\delta_a$.\\
This completes the proof.
\end{proof}
\begin{Lem}\label{-}
Assume that $(G_1)-(G_3)$ and $(V_3)$  hold. Then for any $u\in
\mathcal{P}_a$, we have $(\Psi_u)''(1)<0$. And it is a natural
constraint of $J|_{S_a}$.
\end{Lem}
\begin{proof}
For any $u\in \mathcal{P}_a$, we have
\begin{equation}\label{c1}
s_1|\nabla_{s_1}u|_{2}^{2}+s_2|\nabla_{s_2}u|_{2}^{2}-\displaystyle
\int_{\mathbb{R}^{d}}W(x)u^2
dx=d\int_{\mathbb{R}^{d}}\widetilde{G}(u)dx.
\end{equation}
On the other hand, a direct computation shows that
\begin{align}\label{c2}
\Psi''_{u}(1)&=(2s_{1}-1)s_1|\nabla_{s_1}u|_{2}^{2}+(2s_{2}-1)s_2|\nabla_{s_2}u|_{2}^{2}+\displaystyle
\int_{\mathbb{R}^{d}} W(x) u^{2}dx+\displaystyle
\int_{\mathbb{R}^{d}} \langle \nabla W(x),x\rangle
u^{2}dx\\
&+d(d+1)\displaystyle
\int_{\mathbb{R}^{d}}\widetilde{G}(u)dx-\frac{d^{2}}{2}\displaystyle
\int_{\mathbb{R}^{d}}\widetilde{G}^{'}(u)udx\nonumber.
\end{align}
It follows, using $(G_3)$, $(V_3)$ and \eqref{c2}, that
\begin{align}\label{c3}
\Psi''_{u}(1)&\leq
(2s_{1}-1)s_1|\nabla_{s_1}u|_{2}^{2}+(2s_{2}-1)s_2|\nabla_{s_2}u|_{2}^{2}+\displaystyle
\int_{\mathbb{R}^{d}} W(x) u^{2}dx+\displaystyle
\int_{\mathbb{R}^{d}} \langle \nabla W(x),x\rangle
u^{2}dx\\
&+[(d+1)-\frac{d\alpha}{2}]d\displaystyle
\int_{\mathbb{R}^{d}}\widetilde{G}(u)dx\nonumber\\
&=
(2s_{1}-1)s_1|\nabla_{s_1}u|_{2}^{2}+(2s_{2}-1)s_2|\nabla_{s_2}u|_{2}^{2}+\displaystyle
\int_{\mathbb{R}^{d}} W(x) u^{2}dx+\displaystyle
\int_{\mathbb{R}^{d}} \langle \nabla W(x),x\rangle
u^{2}dx\nonumber\\
&+[(d+1)-\frac{d\alpha}{2}][s_1|\nabla_{s_1}u|_{2}^{2}+s_2|\nabla_{s_2}u|_{2}^{2}-\displaystyle
\int_{\mathbb{R}^{d}}W(x)u^2]\nonumber\\
&\leq
[(d+2s_2)-\frac{d\alpha}{2}][s_1|\nabla_{s_1}u|_{2}^{2}+s_2|\nabla_{s_2}u|_{2}^{2}]+\displaystyle
\int_{\mathbb{R}^{d}} \langle \nabla
W(x),x\rangle u^{2}dx\nonumber\\
&+(\frac{d\alpha}{2}-d-1)\displaystyle
\int_{\mathbb{R}^{d}} W(x) u^{2}dx \nonumber\\
&\leq \displaystyle \int_{\mathbb{R}^{d}}Y(x)u^{2}dx-s_1
(\frac{d\alpha}{2}-d-2s_2)[|\nabla_{s_1}u|_{2}^{2}+|\nabla_{s_2}u|_{2}^{2}]\\
&\leq -[s_1
(\frac{d\alpha}{2}-d-2s_2)-\frac{\sigma_3}{s_1}][|\nabla_{s_1}u|_{2}^{2}+|\nabla_{s_2}u|_{2}^{2}]\nonumber\\
&<0\nonumber.
\end{align}
Thus, $\mathcal{P}_{a}^{+}=\mathcal{P}_{a}^{-}=\emptyset$.
Furthermore, by Proposition \ref{psi}, one can see that it is a
natural constraint of $J|_{S_a}$.\\
This ends the proof.
\end{proof}
\begin{Cor}\label{cor1}
Suppose that $(G_1)-(G_3)$ and $(V_1)-(V_3)$ hold. Then:\\
(i) For any $u\in H^{s_1,s_2}(\R^{d})\setminus \{0\}$, there exists
an unique $t_u>0$ such that $t_u*u\in \mathcal{P}$. Moreover,
\[\Psi_{u}(t_u)=J(t_u*u)=\max\limits_{t>0}J(t*u).\]
(ii) We have
\[C_a:=\inf\limits_{u\in\mathcal{P}_a}J(u)=\inf\limits_{u\in S_a}\max\limits_{t>0} J(t*u)>0.\]
\end{Cor}
\begin{proof}
(i) Let $a:=|u|_{2}^{2}$, then $|\nabla _{s_{1}}u|_{2},|\nabla
_{s_{2}}u|_{2}>0$. From \eqref{m1} and assumptions $(G_1)-(G_2)$,
for $t$ small enough, we find that
\begin{align*}
\Psi'_{u}(t)&=s_1t^{2s_1-1}|\nabla_{s_1}u|_{2}^{2}+s_2t^{2s_2-1}|\nabla_{s_2}u|_{2}^{2}-\frac{1}{t}\displaystyle
\int_{\mathbb{R}^{d}}W(\frac{x}{t})u^{2}dx\\
&-\frac{d}{t^{d+1}}\displaystyle
\int_{\mathbb{R}^{d}}\widetilde{G}(t^{\frac{d}{2}}u(x))dx\\
&\geq
(s_1-\sigma_2)[t^{2s_1-1}|\nabla_{s_1}(u)|_{2}^{2}+t^{2s_2-1}|\nabla_{s_2}(u)|_{2}^{2}]\\
&-(\frac{\beta}{2}-1)\frac{d}{t^{d+1}}\displaystyle
\int_{\mathbb{R}^{d}}G(t^{\frac{d}{2}}u(x))dx \nonumber\\
&\geq (s_1-\sigma_2)[t^{2s_1-1}|\nabla_{s_1}(u)|_{2}^{2}+t^{2s_2-1}|\nabla_{s_1}(u)|_{2}^{2}]\\
&-(\frac{\beta}{2}-1)\frac{d}{t^{-\frac{\alpha d}{2}+
d+1}}\displaystyle \int_{\mathbb{R}^{d}}G(u(x))dx>0.
\end{align*}
Here, we used the fact that, $\alpha>2+\frac{4s_2}{d}$. Hence, there
exists some $t_0>0$ such that $\Psi_u$ increases in $(0,t_0)$. On
the other hand, for $t$ large enough, we have
\begin{align*}
\Psi_u(t)&=\frac{t^{2s_1}}{2}|\nabla_{s_1}u|_2^2+\frac{t^{2s_2}}{2}|\nabla_{s_2}u|_2^2+\frac{1}{2}\int_{\R^d}V(\frac{x}{t})u^2dx
-t^{-d}\int_{\R^d}G(t^{\frac{d}{2}}u(x))dx\\
&\leq\frac{t^{2s_1}}{2}|\nabla_{s_1}u|_2^2+\frac{t^{2s_2}}{2}|\nabla_{s_2}u|_2^2+\frac{1}{2}\int_{\R^d}V(\frac{x}{t})u^2dx
-t^{\frac{d\alpha}{2}-d}\int_{\R^d}G(u(x))dx,
\end{align*}
which implies that $\displaystyle \lim_{t\rightarrow +\infty}
\Psi_{u}(t)=-\infty.$ So there exists some $t_1>t_0$ such that
$$\Psi_{u}(t_1)=\displaystyle \max_{t>0}\Psi_{u}(t).$$
Hence, $\Psi'_{u}(t_1)=0$. It follows, from Proposition
\ref{psi}, that $t_1*u\in \mathcal{P}.$ \\
In what follows we show that $t_1$ is the unique maximum of
$\Psi_u$. Otherwise, suppose that there exists another $t_2>0$ such
that $t_2*u\in \mathcal{P}$. Then, by Lemma \ref{-}, we have that
$t_1$ and $t_2$ are two strict local maxima of $\Psi_u$. Without
loss of generality, we assume that $t_1<t_2$. Then there exists some
$t_3\in (t_1,t_2)$ such that
$$\Psi_{u}(t_3)=\displaystyle \min_{t\in [t_1,t_2]}\Psi_u(t).$$
So, $\Psi'_{u}(t_3)=0$ and $\Psi'_{u}(t_3)\geq 0$. This
 reaches to a contradiction with Lemma
\ref{-}.\\
(ii) From assertion (i), it follows directly that
\[C_a:=\inf\limits_{u\in\mathcal{P}_a}J(u)=\inf\limits_{u\in S_a}\max\limits_{t>0} J(t*u)>0.\]
It remains to show that $C_a>0.$ Let $u\in \mathcal{P}_a.$ In light
of assumptions $(G_1)-(G_3)$ and $(V_2)$, we have
\begin{align}\label{cc1}
d\frac{\alpha-2}{2}\displaystyle\int_{\mathbb{R}^{d}}G(u)dx&\leq d
\displaystyle\int_{\mathbb{R}^{d}}\widetilde{G}(u)dx=s_1|\nabla_{s_{1}}u|_{2}^{2}+s_2|\nabla_{s_{2}}u|_{2}^{2}-\displaystyle\int_{\mathbb{R}^{d}}W(x)u^{2}dx\\
&\leq
(s_1+\sigma_{2})|\nabla_{s_{1}}u|_{2}^{2}+(s_2+\sigma_{2})|\nabla_{s_{2}}u|_{2}^{2}.\nonumber
\end{align}
Then, using \eqref{cc1}, we infer that
\begin{align}\label{cc2}
J(u)&=\frac{1}{2}|\nabla_{s_1}u|_2^2+\frac{1}{2}|\nabla_{s_2}u|_2^2+\frac{1}{2}\int_{\R^d}V(x)u^2dx-\int_{\R^d}G(u)dx\\
&\geq
\frac{1-\sigma_1}{2}(|\nabla_{s_{1}}u|_{2}^{2}+|\nabla_{s_{2}}u|_{2}^{2})-\int_{\R^d}G(u)dx\nonumber\\
&\geq
(\frac{1-\sigma_1}{2}-\frac{2(s_1+\sigma_2)}{d(\alpha-2)})|\nabla_{s_{1}}u|_{2}^{2}+(\frac{1-\sigma_1}{2}-\frac{2(s_2+\sigma_2)}{d(\alpha-2)})(|\nabla_{s_{1}}u|_{2}^{2}
+|\nabla_{s_{2}}u|_{2}^{2})|\nabla_{s_{2}}u|_{2}^{2}\nonumber\\
&\geq
(\frac{1-\sigma_1}{2}-\frac{2(s_1+\sigma_2)}{d(\alpha-2)})(|\nabla_{s_{1}}u|_{2}^{2}+|\nabla_{s_{2}}u|_{2}^{2})\nonumber.
\end{align}
On the other hand, from conditions $(V_1)$ and $(V_2)$, we can
deduce that
\begin{equation}\label{cc3}
\frac{1-\sigma_1}{2}-\frac{2(s_1+\sigma_2)}{d(\alpha-2)}>0 \ \
\mbox{and} \ \
\frac{1-\sigma_1}{2}-\frac{2(s_2+\sigma_2)}{d(\alpha-2)}>0.
\end{equation}
It follows, from \eqref{cc2}, \eqref{cc3} and Lemma \ref{lem4.2},
that
$$C_a\geq (\frac{1-\sigma_1}{2}-\frac{2(s_1+\sigma_2)}{d(\alpha-2)})
\delta_{a}^{2}>0.
$$ This completes the proof.
\end{proof}
\begin{Cor}\label{cor2}
Suppose that the assumptions of Corollary \ref{cor1} are fulfilled.
Then, $J|_{\mathcal{P}_a}$ is coercive, i.e.,
\[\lim\limits_{\substack{u\in \mathcal{P}_a,\\ |\nabla_{s_1} u|_{2}^{2}+|\nabla_{s_2} u|_{2}^{2}\rightarrow \infty}} J(u)=+\infty.\]
\end{Cor}
\begin{proof}
 By the same argument used in the proof of Corollary \ref{cor1}, we
have
\begin{align*}
J(u)\geq
\left(\frac{1-\sigma_1}{2}-\frac{2(s_1+\sigma_2)}{d(\alpha-2)}\right)\left(|\nabla_{s_{1}}u|_{2}^{2}+|\nabla_{s_{2}}u|_{2}^{2}\right)
\rightarrow +\infty,
\end{align*}
as  $u\in \mathcal{P}_{a}$ and $|\nabla_{s_1}
u|_{2}^{2}+|\nabla_{s_2} u|_{2}^{2}\rightarrow
\infty$.\\
 This ends the proof.
\end{proof}

\subsection{Proof of  Theorem \ref{mainnthm} concluded}

Now, we are ready to conclude the proof of Theorem \ref{mainnthm}. We
start by proving point $(1)$, that is, the nonexistence of
normalized solution.
\begin{Pro}\label{mainthm}
Suppose that assumptions $(G_1)-(G_3)$ and $(V_1)-(V_3)$ are
satisfied.  Then, problem \eqref{maineq} has no nontrivial solution
$u\in H^{s_1,s_2}(\R^{d})$ for $\lambda \leq 0$.
\end{Pro}
\begin{proof}
Let $u$ be a solution of \eqref{maineq}. Thus, multiplying
\eqref{maineq} by $u$ and $x\cdot \nabla u$ respectively, and then
integrating by parts, we get
 \begin{equation}\label{p1}
|\nabla_{s_{1}}u|_{2}^{2}+|\nabla_{s_{2}}u|_{2}^{2}+\displaystyle
\int_{\mathbb{R}^{d}}V(x)|u|^{2}dx+\lambda
|u|_{2}^{2}-\int_{\mathbb{R}^{d}}g(u)udx=0
 \end{equation}
 and
 \begin{equation}\label{p2}
\frac{2s_1-d}{2}|\nabla_{s_{1}}u|_{2}^{2}+\frac{2s_2-d}{2}|\nabla_{s_{2}}u|_{2}^{2}-\frac{d}{2}\int_{\mathbb{R}^{d}}V(x)|u|^{2}dx-\int_{\mathbb{R}^{d}}W(x)
u^{2}dx-\frac{d\lambda}{2}|u|_{2}^{2}+d\int_{\mathbb{R}^{d}}G(u)dx=0.
 \end{equation}
 From \eqref{p2}, we obtain
 \begin{align*}
 \frac{2s_1-d}{2}|\nabla_{s_{1}}u|_{2}^{2}+\frac{2s_2-d}{2}|\nabla_{s_{2}}u|_{2}^{2}
 &=\frac{d}{2}\int_{\mathbb{R}^{d}}V(x)|u|^{2}dx+\int_{\mathbb{R}^{d}}W(x)
|u|^{2}dx+\frac{d\lambda}{2}|u|_{2}^{2}-d\int_{\mathbb{R}^{d}}G(u)dx\\
&\leq
\sigma_2|\nabla_{s_{1}}u|_{2}^{2}+\sigma_2|\nabla_{s_{2}}u|_{2}^{2}+\frac{d}{2}\int_{\mathbb{R}^{d}}V(x)|u|^{2}dx+\frac{d\lambda}{2}|u|_{2}^{2}
-d\int_{\mathbb{R}^{d}}G(u)dx.
 \end{align*}
 Thus,
 \[\frac{d\lambda}{2}|u|_{2}^{2}\geq d\int_{\mathbb{R}^{d}}G(u)dx-\left(\sigma_2+\frac{d-2s_1}{2}\right)|\nabla_{s_{1}}u|_{2}^{2}
 -\left(\sigma_2+\frac{d-2s_1}{2}\right)|\nabla_{s_{2}}u|_{2}^{2}
 +\frac{d}{2}\int_{\mathbb{R}^{d}}V(x)|u|^{2}dx.\]
By \eqref{p1}, we obtain
 \begin{align*}
 \frac{d\lambda}{2}|u|_{2}^{2}\geq d\int_{\mathbb{R}^{d}}G(u)dx+(s_2-s_1)|\nabla_{s_{2}}u|_{2}^{2}-(s_1-\sigma_2)\int_{\mathbb{R}^{d}}V(x)|u|^{2}dx
 \\+ \left(\sigma_2+\frac{d-2s_1}{2}\right)\lambda |u|_{2}^{2}
 -\left(\sigma_2+\frac{d-2s_1}{2}\right)\int_{\mathbb{R}^{d}}g(u)udx,
 \end{align*}
 Then by $\int_{\mathbb{R}^{d}}V(x)|u|^{2}dx\leq 0$, we have
  \begin{align*}
(s_1-\sigma_2)\lambda|u|_{2}^{2}\geq d\int_{\mathbb{R}^{d}}G(u)dx
 -\left(\sigma_2+\frac{d-2s_1}{2}\right)\int_{\mathbb{R}^{d}}g(u)udx\geq \tau\int_{\mathbb{R}^{d}}G(u)dx,
 \end{align*}
 where $\tau= d-\left(\sigma_2+\frac{d-2s_1}{2}\right)\beta $. By $\sigma_2<s_1-\frac{(\beta-2)d}{2\beta}$, we have $\tau>0$. Therefore, if $\lambda\leq 0$, we have $\int_{\mathbb{R}^{d}}G(u)dx=0$, then $u$ is trivial.
\end{proof}

 \begin{Rek}
Although there are only trivial solution $u\in H^{s_1,s_2}$ of
\eqref{maineq} provided $\lambda\leq 0$. There may be some
nontrivial solutions in $W^{s_1,p}\cap W^{s_2,q}$, for some $p,q\in
(2,\frac{2d}{d-2s_1})$. For example, \cite{Evequoz} showed the
existence result for nonlinear Helmholtz equation by dual
variational methods.
 \end{Rek}

Now, recall that from Corollary \ref{cor1}, we have
\[C_a:=\inf\limits_{u\in\mathcal{P}_a}J(u)=\inf\limits_{u\in S_a}\max\limits_{t>0}
J(t*u)>0.\]
 We consider the following constrained minimizing problem:
\begin{align*}
\inf\{  J(u): |u|_2^2=a, P(u)=0\}.
\end{align*}
Let $\{u_n\}\subset \mathcal{P}_a$ be a minimal sequence of $J$,
i.e. $J(u_n)\rightarrow C_a>0$.\\
\textbf{Claim $1$.} We have $$C_a<m_a.$$ Indeed, invoking Theorem
\ref{attained}, we can assume that $m_a$ is attained by $v_a\in
S_a$. Therefore, in view of $(V_1)$, one can see that there exists
$C>0$ and a domain $\Omega \subset \R^d$, such that,
\[-\int_\Omega V\left(\frac{x}{t_{v_{a}}}\right)|v_a|^2(x)dx\geq C|\Omega|>0,\]
where $t_{v_{a}}$ is given in Corollary \ref{cor1}. Thus, again by
Corollary \ref{cor1}, we infer that
\begin{align*}
C_a&\leq \max_{t>0} J(t*v_a)=J(t_{v_a}*v_a)=I(t_{v_a}*v_a)+\frac{1}{2}\int_{\R^d}V\left(\frac{x}{t_{v_{a}}}\right)|v_a|^2(x)dx\\
&<I(t_{v_a}*v_a)+\frac{1}{2}\int_{\Omega}V(x)|v_a|^2(x)dx
<I(t_{v_a}*v_a)-C|\Omega|\\
&\leq \max_{t>0} I(t*v_a)=I(v_a)=m_a,
\end{align*}
which implies that $C_a<m_a.$ \\
From Corollary \ref{cor2}, we know that $\{u_n\}$ is bounded in
$H^{s_1,s_2}(\R^{d})$, and so, $u_n\rightharpoonup u$ in
$H^{s_1,s_2}(\R^d)$.\\
\textbf{Claim $2$.} The weak limit $u$ is nontrivial, that is, $u\neq 0$.\\
If we assume that $u=0$, we would have by
Br$\acute{\text{e}}$zis-Lieb lemma and assumption $(V_1)$, that
\begin{align*}
I(u_n)=C_a+o(1),\ \
\Psi'_{\infty,u_n}(1)=o(1).
\end{align*}
Then by the uniqueness, there exists $t_n=1+o(1)$, such that
\[t_n*u_n\in \mathcal{P}_{\infty, a}.\]
Thus, we have
\[m_a\leq I(t_n *u_n)=I(u_n)+o(1)=C_a+o(1),\]
which reach to a contradiction with claim $1$. Hence, $ u\neq 0$.\\
From the fact that $\{u_n\}$ is bounded in $H^{s_1,s_2}(\R^d)$, and
$I(u_n)$ is bounded in $\R$, we can deduce that
$\lambda_n:=-\frac{\langle J'(u_n),u_n\rangle}{a}$ is also bounded
in $\R$. Then, we can assume that for some subsequence
$\lambda_{n}\rightarrow \lambda_a\in \R$.\\
\textbf{Claim $3$.} We must have $\lambda_a>0$.\\
 Since $\{u_n\}\subset
\mathcal{P}_a$, we have $P(u_n)=0$, and so,
\begin{align}\label{1}
s_1|\nabla_{s_1}u_n|_2^2+s_2|\nabla_{s_2}u_n|_2^2-
\int_{\R^d}W(x)|u_n|^2dx-d\int_{\R^d}\Big(\frac{1}{2}g(u_n)u_n-G(u_n)\Big)dx=0
\end{align}
On the other hand, in light of $(V_2)$, one has
\begin{align}\label{2}
\frac{\beta-2}{2}d\int_{\R^d} G(u_n)dx&\geq d\int_{\R^d}\Big(\frac{1}{2}g(u_n)u_n-G(u_n)\Big)dx=d\int_{\R^d}\widetilde{G}(u_n)dx\\
&=s_1|\nabla_{s_1}u_n|_2^2+s_2|\nabla_{s_2}u_n|_2^2- \int_{\R^d}W(x)|u_n|^2dx\nonumber\\
&\geq (s_1-\sigma_2)
|\nabla_{s_1}u_n|_2^2+(s_2-\sigma_2)|\nabla_{s_2}u_n|_2^2.\nonumber
\end{align}
Consequently, by combining \eqref{1} and \eqref{2}, we obtain
\begin{align}\label{3}
\lambda_n|u_n|_2^2&=-\langle J'(u_n),u_n\rangle = -|\nabla_{s_1}u_n|_2^2-|\nabla_{s_2}u_n|_2^2-\int_{\R^d}V(x)u_n^2dx+\int_{\R^d}g(u_n)u_ndx\\
&\geq -|\nabla_{s_1}u_n|_2^2-|\nabla_{s_2}u_n|_2^2+\int_{\R^d}g(u_n)u_ndx\nonumber\\
&=\frac{s_2-s_1}{s_1}|\nabla_{s_2}u_n|_2^2-\frac{1}{s_1}\int_{\R^d}W(x)|u_n|^2dx-\frac{d}{s_1}\int_{\R^d}\widetilde{G}(u_n)dx+\int_{\R^d}g(u_n)u_ndx\nonumber\\
&\geq \frac{s_2-s_1}{s_1}|\nabla_{s_2}u_n|_2^2+\Big(\frac{d}{s_1}-\frac{d-2s_1}{2s_1}\beta\Big)\int_{\R^d} G(u_n)dx-\frac{\sigma_2}{s_1}(|\nabla_{s_1}u_n|_2^2+|\nabla_{s_2}u_n|_2^2)\nonumber\\
&\geq
\Big(\frac{2\beta(s_1-\sigma_2)}{(\beta-2)d}-1\Big)|\nabla_{s_1}u_n|_2^2+
\Big(\frac{2\beta(s_2-\sigma_2)}{(\beta-2)d}-1\Big)|\nabla_{s_2}u_n|_2^2.\nonumber
\end{align}
Recall that from assumption $(V_2)$, we have
$$\sigma_2<s_1-\frac{(\beta-2)d}{2\beta}<s_2-\frac{(\beta-2)d}{2\beta}.$$
It follows, from this fact and \eqref{3}, that
\[\lambda_n|u_n|_2^2\geq C(|\nabla_{s_1}u_n|_2^2+|\nabla_{s_2}u_n|_2^2),\]
for some $C>0$. Thus, in light of Lemma \ref{lem4.2} and claim $2$,
there exists $ \delta>0$ such that $\lambda_n a>\delta$ for all
$n\in \N$. This implies that, up to a subsequence, we may assume
$\lambda_n\rightarrow \lambda_a>0$. This proves claim $3$.\\
 Consequently, $u$ solves the following equation
$$(-\Delta)^{s_1}u(x)+(-\Delta)^{s_2}u(x)+\lambda_a u(x)+V(x)u(x)=g(u(x)).$$
\textbf{Claim $4$.} We have $J(u)>0.$\\
 By $(V_2)$, $(G_2)$, $(G_3)$ and
Lemma \ref{P}, one has
\begin{align}\label{4}
(s_1+\sigma_2)|\nabla_{s_1}u |_2^2+(s_2+\sigma_2)|\nabla_{s_2}u |_2^2&\geq s_1|\nabla_{s_1}u |_2^2+s_2|\nabla_{s_2}u |_2^2-\int_{\R^d}W(x)|u(x)|^2dx\\
&=d\int_{\R^d}\widetilde{G}(u)dx\nonumber\\
&\geq \frac{d(\alpha-2)}{2}\int_{\R^d}G(u)dx.\nonumber
\end{align}
Again, from $(V_1)$ and $(V_2)$, we can remark that
\begin{equation}\label{5}
\frac{1-\sigma_1}{2}>\frac{2(s_2+\sigma_2)}{d(\alpha-2)}>\frac{2(s_1+\sigma_2)}{d(\alpha-2)}.
\end{equation}
Then, in view of \eqref{4}, \eqref{5} and $(V_1)$ , we obtain
\begin{align*}
J(u)&=\frac{1}{2}|\nabla_{s_1}u|_2^2+\frac{1}{2}|\nabla_{s_2}u|_2^2+\frac{1}{2}\int_{\R^d}V(x)u^2dx-\int_{\R^d}G(u)dx\\
&\geq \frac{1-\sigma_1}{2}(|\nabla_{s_1}u|_2^2+ |\nabla_{s_2}u|_2^2)-\int_{\R^d}G(u)dx\\
&\geq
\Big(\frac{1-\sigma_1}{2}-\frac{2(s_1+\sigma_2)}{d(\alpha-2)}\Big)|\nabla_{s_1}u|_2^2
+\Big(\frac{1-\sigma_1}{2}-\frac{2(s_2+\sigma_2)}{d(\alpha-2)}\Big)|\nabla_{s_2}u|_2^2.
\end{align*}
This proves, from claim $2$ and Lemma
\ref{lem4.2} that $J(u)>0$. \\
\textbf{Claim $5$.} We show that $|u|_2^2=a$.\\ Assume that
$b:=|u|_2^2 \neq a$. Put $c:=a-b$, due to the norm is weakly lower
semi-continuous, we have $c\in (0,a)$. Set $\phi_n:=u_n-u$. Then, by
$(V_1)$, we get
\[\int_{\R^d} V(x)|\phi_n|^2dx\rightarrow 0, \ \text{as}\ n\rightarrow \infty.\]
Moreover, by using the Br$\acute{\text{e}}$zis-Lieb lemma, we have $|\phi_n|_2^2=c+o(1)$, $I(\phi_n)+J(u)=C_a+o(1)$, and $\Psi_{\infty, \phi_n}'(1)=o(1)$. If $\liminf\limits_{n\rightarrow \infty}
(|\nabla_{s_1}\phi_n|_2^2+|\nabla_{s_2}\phi_n|_2^2)=0$, then by $$s_1|\nabla_{s_1}\phi_n|_2^2+s_2|\nabla_{s_2}\phi_n|_2^2=d\int_{\R^d}\widetilde{G}(\phi_n)dx+o(1),$$ we have $\liminf\limits_{n\rightarrow \infty}\int_{\R^d}\widetilde{G}(\phi_n)dx=0$. Similar to \eqref{3}, by using the Br$\acute{\text{e}}$zis-Lieb lemma, we have
\[\lambda_n |u_n|_2^2-\lambda_a|u|_2^2= -|\nabla_{s_1}\phi_n|_2^2-|\nabla_{s_2}\phi_n|_2^2 +\int_{\R^d}g(\phi_n)\phi_ndx+o(1).\]
Thus, we have $\lambda_a c=0$, which is impossible. Therefore, we have
\[\liminf_{n\rightarrow \infty}
(|\nabla_{s_1}\phi_n|_2^2+|\nabla_{s_2}\phi_n|_2^2)>0.\] By the uniqueness, there exists $t_n=1+o(1)$ such
that $t_n*\phi_n\in \mathcal{P}_{\infty, |\phi_n|_2^2}$. Then, by
Lemmas \ref{decc} and \ref{con}, we have
\[m_{|\phi_n|_2^2}\leq I(t_n*\phi_n)=I(\phi_n)+o(1)=C_a-J(u)+o(1).\]
Therefore,
\[m_a\leq m_c\leq \lim\limits_{n\rightarrow \infty} m_{|\phi_n|_2^2}\leq C_a-J(u)\leq m_a-J(u).\]
Hence, using claim $4$, we reach to a contradiction.\\
Claim $5$ leads to the fact that  $u_n\rightarrow u$ in $L^2(\R^d)$,
and $u\in \mathcal{P}_a$.  By the fractional Gagliardo-Nirenberg inequality, see \cite{hajeiej2}, we have
\[ |u_n-u|_p\leq B^\frac{A}{p}|\nabla_{s_1}u_n-\nabla_{s_1}u|_2^{\frac{A}{p}}\cdot |u_n-u|_2^{\frac{p-A}{p}},  \]
where $A=\frac{(p-2)d}{2s_1}$, $1\leq p<\frac{2d}{d-2s_1}$, and $B$ is the best constant give by
\[B=2^{-2s_1}\pi^{-s_1} \frac{\Gamma((d-2s_1)/2)}{\Gamma((d+2s_1)/2)}\left(\frac{\Gamma(d)}{\Gamma(d/2)}\right)^{2s_1/d}.\]
Since $2+\frac{4s_2}{d}<\alpha<\beta<\frac{2d}{d-2s_1}$, we have
 $\int_{\R^d}G(u_n)dx
\rightarrow \int_{\R^d}G(u)dx$.  Thus,
\[C_a\leq J(u)\leq \liminf_{n\rightarrow \infty} J(u_n)=C_a,\]
which implies $J(u)=C_a$. Moreover,
$|\nabla_{s_1}u_n|_2^2\rightarrow |\nabla_{s_1}u |_2^2$,
$|\nabla_{s_2}u_n|_2^2\rightarrow |\nabla_{s_2}u |_2^2$, as
$n\rightarrow \infty$. This shows that $u_n\rightarrow u$ in $H^{s_1,s_2}(\R^{d})$, and $u\in \mathcal{P}_a$ attains $C_a$.
\\
Then by Proposition \ref{p4.3}, there exists $\lambda\in
\R^{+}$ such that $(\lambda, u)$ solves
 \[(-\Delta)^{s_1}u(x)+(-\Delta)^{s_2}u(x)+\lambda u(x)+V(x)u(x)=g(u(x)),\]
with $\int_{\mathbb{R}^d} |u(x)|^2dx=a$.\\
This ends the proof.

\section{Proof of Theorem \ref{mainnthm2}}
We show the proof of Theorem \ref{mainnthm2} in this section. Fist, we introduce some notations here.
Recall that
\[J(u)=\frac{1}{2}|\nabla_{s_1}u|_2^2+\frac{1}{2}|\nabla_{s_2}u|_2^2+\frac{1}{2}\int_{\R^d}V(x)u^2dx-\int_{\R^d}G(u)dx.\]
This functional is defined on the following Hilbert space
\[H_V^{s_1,s_2}(\R^d):=\{u\in H^{s_1,s_2}(\R^d): \int_{\R^d} V(x)|u|^2dx<\infty\},\]
with the norm
\[\|u\|_V^2:=\int_{\R^d}\left(|\nabla_{s_1}u|^2+|\nabla_{s_2}u|^2+ V(x)|u|^2\right)dx,\]
which is induced by the following inner product
\[\langle u,v\rangle_V:=\frac{1}{2}\sum\limits_{i=1}^{2}\int_{\R^d\times \R^d} \frac{(u(x)-u(y))v(x)}{|x-y|^{d+2s_i}}dxdy+\int_{\R^d}  V(x)u(x)  v(x) dx,\]
Thus we have
\[J(u)=\frac{1}{2}\|u\|_V-\int_{\R^d}G(u)dx, \ J'(u)v=\langle u,v\rangle_V-\int_{\R^d}g(u)vdx.\]

\begin{Lem}\label{cptemb}
Assume $(V_1')$, $(V_2')$ are satisfied, then $H_V^{s_1,s_2}(\R^d)$ is continuously embedded in $L^p(\R^d)$, for any $p\in [2,2^*_{s_2}]$. Moreover, $H_V^{s_1,s_2}(\R^d)$ is compactly embedded in $L^p(\R^d)$, for any $p\in [2,2^*_{s_2})$.
\end{Lem}
Equivalently, we consider the following constrained minimizing problem:
\[E_a=\inf\{J(u):u\in S_a\}.\]
Let $\{u_n\}\subset S_a$ be a minimizing sequence of $J$ with respect to $E_a$, i.e. $J(u_n)\rightarrow E_a$, and $|u_n|_2^2=a$. \\
\textbf{Claim $1$.} $\{u_n\}$ is bounded in $H_V^{s_1,s_2}(\R^d)$.
\\
Thus, going to a subsequence if necessary, $u_n\rightharpoonup u$ in
$H_V^{s_1,s_2}(\R^d)$. By Lemma \ref{cptemb}, we can extract a subsequence such that $u_n\rightarrow u$ in $L^p(\R^d)$, $2\leq p<2^*_{s_2}$. \\
\textbf{Claim $2$.}
$\lim\limits_{n \rightarrow +\infty} \int_{\R^{d}}G(u_n)dx=\int_{\R^{d}}G(u)dx,$
and
$\lim\limits_{n \rightarrow +\infty} \int_{\R^{d}}g(u_n) u_ndx=\int_{\R^{d}}g(u)udx.$
\\
Therefore, we have
\[J(u)\leq \liminf_{n\rightarrow \infty}J(u_n)=E_{a },\ |u|_2^2=\lim_{n\rightarrow \infty}|u_n|_2^2=a.\]
Thus, $u\in S_a$, and $J(u)=E_a$, which means $E_a$ is attained. Let
$$\lambda_n=-\frac{1}{a}\langle J'(u_n),u_n\rangle=\frac{1}{a}\left(\int_{\R^d} g(u_n)u_n dx-\|u_n\|_V\right).$$
Similar to the proof of Theorem \ref{mainnthm}, we have $\{\lambda_n\}$ is bounded and for some subsequence $\lambda_n\rightarrow \lambda_a>0$. Take $\lambda=\lambda_a$, we have $(\lambda ,u)$ solves
 \[(-\Delta)^{s_1}u(x)+(-\Delta)^{s_2}u(x)+\lambda  u(x)+V(x)u(x)=g(u(x)),\]
with $\int_{\mathbb{R}^d} |u(x)|^2dx=a$.\\
This ends the proof.


\end{document}